\documentclass[a4paper,10pt]{article}
\usepackage[utf8x]{inputenc}

\usepackage{nccfoots}
\usepackage{times}
\usepackage{geometry}
\usepackage[T1]{fontenc}
\usepackage{color}
\usepackage{amsmath}
\usepackage{amssymb}
\usepackage{array}
\usepackage{amsthm}
\usepackage{graphicx}
\usepackage{hyperref}
\usepackage{setspace}
\usepackage{stmaryrd}
\usepackage{marginnote}
\usepackage{esint}

\theoremstyle{plain}
\newtheorem{thm}{Theorem}[section]
\newtheorem{prop}[thm]{Proposition}
\newtheorem{cor}[thm]{Corollary}
\newtheorem{lem}[thm]{Lemma}

\theoremstyle{definition}
\newtheorem{defn}[thm]{Definition}

\theoremstyle{remark}
\newtheorem{rem}[thm]{Remark}

\theoremstyle{plain}

\newcommand{\R}{\mathbb{R}}

\newcommand{\C}{\mathbb{C}}

\newcommand{\CP}{\mathbb{C}P}
\newcommand{\HP}{\mathbb{H}P}

\newcommand{\grad}{\mathrm{grad}}

\newcommand{\identity}{\mathrm{id}}
\newcommand{\scal}{\mathrm{scal}}
\newcommand{\ric}{\mathrm{Ric}}
\newcommand{\trace}{\mathrm{tr}}
\newcommand{\kernel}{\mathrm{ker}}

\newcommand{\volume}{\mathrm{vol}}

\newcommand{\dv}{\text{ }dV}

\newcommand{\Diff}{\mathrm{Diff}}

\newcommand{\spectrum}{\mathrm{spec}}
\newcommand{\gradient}{\mathrm{grad}}

\newcommand*{\suchthat}[1]{\left|\vphantom{#1}\right.}

\renewcommand{\title}[1]{{\bfseries #1}\par}
\renewcommand{\author}[1]{\medskip{#1}\par\smallskip}
\newcommand{\affiliation}[1]{{\itshape #1}\par}
\newcommand{\email}[1]{E-mail:~\texttt{#1}\par}

\numberwithin{equation}{section}

\begin{document}
\begin{center}
\title{\LARGE Stability of Einstein metrics under Ricci flow}
\vspace{3mm}
\author{\Large Klaus Kröncke}
\vspace{3mm}
\affiliation{Universität Hamburg, Fachbereich Mathematik\\Bundesstraße 55\\20146 Hamburg, Germany}
\email{klaus.kroencke@uni-hamburg.de} 
\end{center}
\vspace{2mm}
\begin{abstract}We prove dynamical stability and instability theorems for compact Einstein metrics under the Ricci flow.
We give a nearly complete charactarization of dynamical stability and instability in terms of the conformal Yamabe invariant and the Laplace spectrum.
In particular, we prove dynamical stability of some classes of Einstein manifolds for which it was previously not known. Additionally, we show that the complex projective space with the Fubini-Study metric is surprisingly dynamically unstable.
\end{abstract}
 \section{Introduction}
Let $M^n$, $n\geq2$ be a manifold. A Ricci flow on $M$ is a curve of metrics $g(t)$ on $M$ satisfying the evolution equation
 \begin{align}\label{ricciflow}\dot{g}(t)=-2\ric_{g(t)}.
 \end{align}
The Ricci flow was first introduced by Hamilton in \cite{Ham82}.
Since then, it has become an important tool in Riemannian geometry. It was not only an essential tool in the proof of the famous Poincare conjecture \cite{Per02,Per03} but also for proving other recent results like the differentiable sphere theorem \cite{BS09}.

The Ricci flow is not a gradient flow in the strict sence, but Perelman made the remarkable discovery that it can be interpreted as the gradient of the $\lambda$-functional
\begin{equation}\label{perelmanslambda}\lambda(g)=\inf_{\substack{f\in C^{\infty}(M)\\\int_M e^{-f}\dv_g=1}}\int_M(\scal_g+|\nabla f|^2_g)e^{-f}\dv_g\end{equation}
 on the space of metrics modulo diffeomorphisms \cite{Per02}.

\Footnotetext{}{2010 \emph{Mathematics Subject Classification.} 53C25,53C44.}
\Footnotetext{}{\emph{Key words and phrases.} Einstein metrics, Ricci flow, dynamical stability, Yamabe invariant.}

Ricci-flat metrics are the stationary points of the Ricci flow and Einstein metrics remain unchanged under the Ricci flow up to rescaling. It is now natural to ask how the Ricci flow behaves as a dynamical system close to Einstein metrics.
A stability result for compact Einstein metrics assuming positivity of the Einstein operator was proven in \cite{Ye93}.
Stability results for compact Ricci-flat metrics assuming nonnegativity of the Lichnerowicz Laplacian and integrability of infinitesimal Einstein deformations were proven by Sesum and Haslhofer in \cite{Ses06,Has12}, generalizing an
 older result in \cite{GIK02}.
Recently, Haslhofer and M\"uller \cite{HM14} were able to get rid of the integrability condition and proved the following:
%\begin{thm}[{{\cite[Theorem 1.1]{HM14}}}]\label{ricciflatstability}Let $(M,g_{RF})$ be a compact Ricci-flat manifold. If $g_{RF}$ is a local maximizer of $\lambda$, then for every $C^{k,\alpha}$-neighbourhood $\mathcal{U}$ of $g_{RF}$ there exists a
%$C^{k,\alpha}$-neighbourhood $\mathcal{V}$ such that the Ricci flow starting at any metric in $\mathcal{V}$ exists for all time and converges (modulo diffeomorphism\index{modulo diffeomorphism}) to a Ricci-flat\index{Ricci-flat} metric in $\mathcal{U}$.
%\end{thm}
%\begin{thm}[{{\cite[Theorem 1.2]{HM14}}}]\label{ricciflatinstability}Let $(M,g_{RF})$ be a compact Ricci-flat metric. If $g_{RF}$ is not a local maximizer of $\lambda$, then there exists an ancient Ricci flow\index{ancient} $g(t)$, $t\in(-\infty,0]$ which converges modulo diffeomorphism to $g_{RF}$ as $t\to-\infty$.
% \end{thm}
A compact Ricci-flat manifold is dynamically stable if it is a local maximizer of $\lambda$ and dynamically unstable, if this is not the case. Because of monotonicity of $\lambda$ along the Ricci flow, the converse implications hold in both cases.

The aim of the present paper is to generalize these results to the Einstein case and to give geometric stability and instability conditions in terms of the conformal Yamabe invariant and the Laplace spectrum.
Throughout, any manifold will be compact.
 The Yamabe invariant of a conformal class is defined by
\begin{align*}Y(M,[g])=\inf_{\tilde{g}\in[g]}\volume(M,\tilde{g})^{2/n-1}\int_M\scal_{\tilde{g}}\dv_{\tilde{g}},
 \end{align*}
where $[g]$ denotes the conformal class of the metric $g$.
By the solution of the Yamabe problem, this infimum is always realized by a metric of constant scalar curvature \cite{Sch84}. Let $M$ be a manifold and $\mathcal{M}$ be the set of smooth metrics on $M$.
We call the map $\mathcal{M}\ni g\mapsto Y(M,[g])$ the Yamabe functional and the real number $$Y(M)=\sup_{g\in\mathcal{M}}Y(M,[g])$$ the smooth Yamabe invariant of $M$.
A metric $g$ on $M$ is called supreme if it realizes the conformal Yamabe invariant in its conformal class and the smooth Yamabe invariant of the manifold.

It is a hard problem to compute the smooth Yamabe invariant of a given compact manifold and only for a few examples (including the round sphere), it is explicitly known. An interesting question is whether
a compact manifold admits a supreme metric and whether it is Einstein. For more details concerning these questions, see e.g.\ \cite{LeB99}.

Any Einstein metric $g_E$ is a critical point of the Yamabe functional and it is a local maximum of the Yamabe functional if and only if $g_E$ is a local maximum of the Einstein-Hilbert action restricted to the set of constant
scalar curvature metrics of volume $\volume(M,g_E)$. This follows from \cite[Theorem C]{BWZ04}.
A sufficient condition for this is that the Einstein operator
\begin{align*}\Delta_E=\nabla^*\nabla-2\mathring{R}
 \end{align*}
is positive on all nonzero transverse traceless tensors, i.e.\ the symmetric $(0,2)$-tensors satisfying $\trace h=0$ and $\delta h=0$ (\cite[p.\ 279]{Boe05} and \cite[p.\ 131]{Bes08}).
Here, $\mathring{R}$ denotes the natural action of the curvature tensor on symmetric $(0,2)$-tensors.
 Conversely, if $g_E$ is a local maximum of the Yamabe functional, the Einstein operator is nessecarily nonnegative on transverse traceless tensors.
The Einstein operator and its spectrum were studied in \cite{Koi78,Koi83,IN05,DWW05,DWW07} and also in a recent paper by the author \cite{Kro15}. We find the following relation to the $\lambda$-functional which will be proven in Section \ref{localmax}.
\begin{thm}\label{lambdaandyamabe}
A Ricci-flat metric $g_{RF}$ is a local maximizer of $\lambda$ if and only if
it is a local maximizer of the Yamabe functional, i.e.\ there are no metrics of positive scalar curvature close to $g_E$.
\end{thm}
\begin{rem}
%By the solution of the Yamabe problem we see how to understand the maximality of $\lambda$ geometrically: A Ricci flat metric $g_{RF}$ is a local maximizer of $\lambda$ if and only if there are no metrics of
%positive scalar curvature close to $g_{RF}$.
%Let us give a short argument why these conditions are equivalent: If $g_{RF}$ admits positive scalar curvature metrics in its neighbourhood, then $\lambda$ is positive on these metrics by \ref{perelmanslambda}. On the other hand, if this is not the case, then by the solution of the Yamabe problem \cite{Sch84}, the smallest eigenvalue of the Yamabe operator $4\frac{n-1}{n-2}\Delta+\scal$ is nonpositive for all metrics close to $g_{RF}$ and so the same holds for $\lambda$, since it is the smallest eigenvalue of $4\Delta+\scal$.
 This condition is automatic if $M$ is spin and if $\hat{A}(M)\neq0$ \cite[p.\ 46]{Hit74} because the existence of positive scalar curvature metrics is excluded.
\end{rem}
Since Einstein metrics are not stationary points of the Ricci flow in its original form we consider the volume-normalized Ricci flow
\begin{align}
\dot{g}(t)=-2\ric_{g(t)}+\frac{2}{n}\left(\fint_M\scal_{g(t)}\dv_{g(t)}\right)\cdot g(t).
\end{align}
This allows us to define appropriate notions of dynamical stability and instability for Einstein metrics.
\begin{defn}
 A compact Einstein manifold $(M,g_E)$ is called dynamically stable if for any $k\geq3$ and any $C^k$-neighbourhood $\mathcal{U}$ of $g_E$ in the space of metrics, there exists a $C^{k+2}$-neighbourhood $\mathcal{V}\subset\mathcal{U}$
such that for any $g_0\in\mathcal{V}$, the normalized Ricci flow starting at $g_0$  exists for all $t\geq0$ and converges modulo diffeomorphism to an Einstein metric in $\mathcal{U}$ as $t\to\infty$.

We call a compact Einstein manifold $(M,g_E)$ dynamically unstable if there exists a nontrivial normalized Ricci flow defined on $(-\infty,0]$ which converges modulo diffeomorphism to $g_E$ as $t\to-\infty$.
\end{defn}
It is well-known that the round sphere is dynamically stable \cite{Ham82,Hui85}. From now on, we assume that $(M^n,g)\neq (S^n,g_{st})$ and that $n\geq 3$. The round sphere is an exceptional case because 
it is the only compact Einstein space which admits conformal Killing vector fields.
Now we can state the main theorems of this paper.
\begin{thm}[Dynamical stability]\label{Thm1}Let $(M,g_E)$ be a compact Einstein manifold with Einstein constant $\mu$. Suppose that $(M,g_E)$ is a local maximizer of the Yamabe functional and if the smallest nonzero eigenvalue of the Laplacian satisfies $\lambda>2\mu$.
 Then $(M,g_E)$ is dynamically stable.
\end{thm}
\begin{thm}[Dynamical instability]\label{Thm2}Let $(M,g_E)$ be a compact Einstein manifold with Einstein constant $\mu$. Suppose that $(M,g_E)$ is a not local maximizer of the Yamabe functional or the smallest nonzero eigenvalue of the Laplacian satisfies $\lambda<2\mu$.
 Then $(M,g_E)$ is dynamically unstable.
\end{thm}
Apart from the case $\lambda=2\mu$, this gives a complete description of the Ricci flow as a dynamical system close to a compact Einstein metric. The converse implications nearly hold:
If an Einstein manifold is dynamically stable, then it is a local maximizer of the Yamabe functional and the smallest nonzero eigenvalue of the Laplacian satisfies $\lambda\geq2\mu$. If it is dynamically unstable
it is not a local maximizer of the Yamabe functional or the smallest nonzero Laplace eigenvalue satisfies $\lambda\leq2\mu$. This follows from Theorem \ref{yamabemaximality} resp.\ Theorem \ref{yamabemaximalityII} and the monotonicity of
the functionals $\mu_+$ and $\nu_-$ along the corresponding variants of the Ricci flow.

The Ricci-flat case is already covered by the results in \cite{HM14} and Theorem \ref{lambdaandyamabe}, so it remains to consider the cases of positive and negative Einstein constant.
Both cases will be proven separately. We use the functional $\mu_+$ in the negative case and $\nu_-$ in the positive case. Both are analogues of the $\lambda$-functional.
The negative case will be
treated more extensively. In the positive case, the strategy is basically the same and so we will skip the details there.

Observe that the case of nonpositive Einstein constant is easier to handle with because the eigenvalue condition drops there.
In fact, all known compact nonpositive Einstein manifolds satisfy the assumptions of Theorem \ref{Thm1}.
For certain classes of nonpositive Einstein manifolds we actually know that these assumptions hold: By \cite[Theorem 3.6]{LeB99}, any $4$-dimensional K\"ahler-Einstein manifold with nonpositive scalar curvature 
realizes the smooth Yamabe invariant of $M$. Thus, we have
\begin{cor}
 Any compact four-dimensional K\"ahler-Einstein manifold with nonpositive scalar curvature is dynamically stable.
\end{cor}
%If the Ricci curvature of a K\"ahler-Einstein metric is zero we do not have to restrict to the four-dimensional case. By the Bogomolov-Tian-Todorov theorem \cite{Bog78,Tia87,Tod89}, the moduli space of Ricci-flat K\"ahler metrics
%is a manifold and thus by \cite[Theorem 1.3]{DWW07} they are local maxima of the Yamabe functional.
%\begin{cor}Any compact Ricci-flat K\"ahler metric is dynamically stable.
% \end{cor}
%In fact the stability of these spaces follows already from the stability theorems in \cite{Ses06,Has12} because the integrability condition is satisfied.
%By \cite[Theorem 4.2]{DWW05}, compact Ricci-flat spin manifolds with a parallel spinor can not be deformed to metrics of positive scalar curvature and thus, they are local maxima of the Yamabe functional.
%\begin{cor} Any compact Ricci-flat spin manifold admitting a parallel spinor is dynamically stable.
% \end{cor}
%In dimensions $n\neq 4m$, it implies that all Calabi-Yau manifolds (i.e. Ricci-flat K\"ahler manifolds) are dynamically stable: By the Cheeger-Gromoll splitting theorem \cite{CG71}, any Calabi-Yau manifold admits a finite cover which is a product of a simply-connected Calabi-Yau manifold with a torus and thus admits a parallel spinor.

%As was already said, no compact dynamically unstable Einstein metrics with nonpositive Einstein constant are known.
 In contrast, there are many positive Einstein metrics, which satisfy one of the conditions of Theorem \ref{Thm2}.
For example, any product of two positive Einstein metrics does not maximize the Yamabe functional because the Einstein operator admits negative eigenvalues. On the other hand, there are some symmetric spaces of compact type (e.g.\ $\HP^n$ for $n\geq3$, see \cite[Table 2]{CH13}), which are local maxima 
of the Yamabe functional and satisfy the eigenvalue condition of Theorem \ref{Thm2}.

However, some interesting examples are not covered by the above two theorems because they are local maxima of the Yamabe functional but the smallest nonzero eigenvalue of the Laplacian satisfies $\lambda=2\mu$.
These examples include the symmetric spaces $G_2,\CP^n,SO(n+2)/(SO(n)\times SO(2)),n\geq5,S(2n)/U(n),n\geq5,E_6/(SO(10)\cdot SO(2)),E_7/(E_6\cdot SO(2))$ with their standard metric, see \cite[Table 1 and Table 2]{CH13}.
For such manifolds, we prove dynamical instability under an additional condition.
\begin{thm}[Dynamical instability]\label{Thm3}Let $(M^n,g_E)$, $n\geq3$ be a compact Einstein manifold with Einstein constant $\mu$. Suppose that there exists a $v\in C^{\infty}(M)$ statisfying $\Delta v=2\mu v$ and $\int_M v^3\dv\neq0$. Then $(M,g_E)$ is dynamically unstable.
 \end{thm}
 \noindent
We construct an eigenfunction on $\CP^n$ with its standard metric satisfying this condition and thus, we have
\begin{cor}\label{unstableCPn}The manifold $(\CP^n,g_{st})$, $n>1$ is dynamically unstable.
 \end{cor}
 This result is quite unexpected since the complex projective space is linearly stable. In particular, it raises the question whether the round sphere is the only positive Einstein metric in four dimensions which is dynamically stable (c.f.\ \cite[p.\ 29]{Cao10}). 
\vspace{3mm}

\textbf{Acknowledgement.} This article is based on a part of the authors' PhD-thesis. The author would like to thank Christian B\"ar, Christian Becker and Robert Haslhofer for helpful discussions. Moreover, the 
author thanks the Max-Planck Institute for Gravitational Physics for financial support.
\section{Notation and conventions}
We define the Laplace-Beltrami operator acting on functions by $\Delta=-\trace \nabla^2$. For the Riemann curvature tensor, we use the sign convention such that
 $R_{X,Y}Z=\nabla^2_{X,Y}Z-\nabla^2_{Y,X}Z$. Given a fixed metric, we equip the bundle of $(r,s)$-tensor fields (and any subbundle) with the natural scalar product induced by the metric.
By $S^pM$, we denote the bundle of symmetric $(0,p)$-tensors.
The divergence $\delta:\Gamma(S^pM)\to\Gamma(S^{p-1}M)$ and its formal adjoint $\delta^*\colon\Gamma(S^{p-1}M)\to \Gamma(S^pM)$ are given by
\begin{align*}\delta T(X_1,\ldots,X_{p-1})=&-\sum_{i=1}^n\nabla_{e_i}T(e_i,X_1,\ldots,X_{p-1}),\\
            \delta^*T(X_1,\ldots,X_p)=&\frac{1}{p}\sum_{i=0}^{p-1}\nabla_{X_{1+i}}T(X_{2+i},\ldots,X_{p+i}),
\end{align*}
where the sums $1+i,\ldots,p+i$ are taken modulo $p$. 
For $\omega\in\Omega^1(M)$, we have $\delta^*\omega=\mathcal{L}_{\omega^{\sharp}}g$ where $\omega^{\sharp}$ is the sharp of $\omega$.
Thus, $\delta^*(\Omega^1(M))$ is the tangent space of the manifold $g\cdot \Diff(M)=\left\{\varphi^*g|\varphi\in\Diff(M)\right\}$.
The Einstein operator $\Delta_E$ and the Lichnerowicz Laplacian $\Delta_L$, both acting on $\Gamma(S^2M)$, are defined by
\begin{align*}
 \Delta_E h&=\nabla^*\nabla h-2\mathring{R}h,\\
 \Delta_L h&=\nabla^*\nabla h+\ric\circ h+h\circ \ric-2\mathring{R}h.
\end{align*}
Here, $\mathring{R}h(X,Y)=\sum_{i=1}^nh(R_{e_i,X}Y,e_i)$ and $\circ$ denotes the composition of symmetric $(0,2)$-tensors, considered as endomorphisms on $TM$.
\section{The expander entropy}
When considering the Ricci flow close to negative Einstein manifolds we may restrict to the case where the Einstein constant is equal to $-1$.
Such metrics are stationary points of the flow
\begin{align}\label{negativeflow}\dot{g}(t)=-2(\ric_{g(t)}+g(t)).
\end{align}
This flow is homothetically equivalent to the standard Ricci flow. In fact,
\begin{align*}\tilde{g}(t)=e^{-2t}g\left(\frac{1}{2}(e^{2t}-1)\right)
\end{align*}
is a solution of \eqref{negativeflow} starting at $g_0$ if and only if $g(t)$ is a solution of \eqref{ricciflow} starting at $g_0$.
Let $(M,g)$ be a Riemannian manifold and $f\in C^{\infty}(M)$. Define\index{$\mathcal{W}_+(g,f)$}
\begin{align*}\mathcal{W}_+(g,f)=\int_M \left[\frac{1}{2}(|\nabla f|^2+\scal)-f\right]e^{-f}\dv.
\end{align*}
This is a simpler variant of the expander entropy\index{expander entropy} $\mathcal{W}_+(g,f,\sigma)$\index{$\mathcal{W}_+(g,f,\sigma)$} introduced in \cite{FIN05}. 
\begin{lem}\label{Wvariation}The first variation of $\mathcal{W}_+$ at a tuple $(g,f)$ equals
\begin{align*}\mathcal{W}'_+(h,v)=\int_M &[-\frac{1}{2}\langle\ric+\nabla^2 f-(-\Delta f-\frac{1}{2}|\nabla f|^2+\frac{1}{2}\scal -f)g,h\rangle\\
                                                          &-(-\Delta f-\frac{1}{2}|\nabla f|^2+\frac{1}{2}\scal-f+1)v]e^{-f}\dv.
\end{align*} 
\end{lem}
\begin{proof}Let $g_t=g+th$ and $f_t=f+tv$. We have
\begin{align*}\frac{d}{dt}|_{t=0}\mathcal{W}_+(g_t,f_t)=&\int_M [\frac{1}{2}(|\nabla f_t|_{g_t}^2+\scal_{g_t})-f_t]'e^{-f}\dv\\
                                                          &+\int_M [\frac{1}{2}(|\nabla f|^2+\scal)-f](-v+\frac{1}{2}\trace h)e^{-f}\dv.
\end{align*}
By the variational formula of the scalar curvature (see Lemma \ref{firstcuvature}),
\begin{align*}\int_M [\frac{1}{2}(|\nabla f_t|_{g_t}^2+\scal_{g_t})-f_t]'e^{-f}\dv
               =&\int_M (-\frac{1}{2}\langle h,\nabla f\otimes \nabla f\rangle+ \langle \nabla f,\nabla v\rangle)e^{-f}\dv\\
                                                                    &+\int_M [\frac{1}{2}(\Delta \trace h+\delta(\delta h)-\langle \ric,h\rangle)-v]e^{-f}\dv.
\end{align*}
By integration by parts,
\begin{align*} \int_M\langle \nabla f,\nabla v\rangle e^{-f}\dv=\int_M (\Delta f+|\nabla f|^2)v e^{-f}\dv\end{align*}
and
\begin{align*}\int_M\frac{1}{2}(\Delta \trace h+\delta(\delta h))e^{-f}\dv&=\int_M \frac{1}{2}[\trace h \Delta(e^{-f})+\langle h,\nabla^2 (e^{-f})\rangle]\dv \\
&=\int_M \frac{1}{2}[\trace h (-\Delta f-|\nabla f|^2)+\langle h,-\nabla^2 f+\nabla f\otimes \nabla f\rangle]e^{-f}\dv.
\end{align*}
Thus, 
\begin{align*}\int_M [\frac{1}{2}(|\nabla f_t|_{g_t}^2+\scal_{g_t})-f_t]'e^{-f}\dv=\int_M &[-\frac{1}{2}\langle h,\nabla^2 f+\ric+(\Delta f+|\nabla f|^2)g\rangle\\
                                                                    & +(\Delta f+|\nabla f|^2-1)v]e^{-f}\dv.
\end{align*}
The second term of above can be written as
\begin{align*}\int_M [\frac{1}{2}&(|\nabla f|^2+\scal)-f](-v+\frac{1}{2}\trace h)e^{-f}\dv\\
=&\int_M [\frac{1}{2}\langle[\frac{1}{2}(|\nabla f|^2+\scal)-f]g,h\rangle-[\frac{1}{2}(|\nabla f|^2+\scal)-f]v ]e^{-f}\dv.
\end{align*}
By adding up these two terms, we obtain the desired formula.
\end{proof}
\noindent
Now we consider the functional\index{$\mu_+(g)$, expander entropy}
\begin{align}\mu_+(g)=\inf\left\{\mathcal{W}_+(g,f)\suchthat{ f\in C^{\infty}(M),\text{ }\int_M e^{-f} \dv=1} f\in C^{\infty}(M),\text{ }\int_M e^{-f} \dv=1\right\}.
\end{align}
\index{mu@$\mu_+$-functional}It was shown in \cite[Thm 1.7]{FIN05} that given any smooth metric, 
the infimum is always uniquely realized by a smooth function. We call the minimizer\index{minimizer} $f_g$\index{$f_g$, minimizer realizing $\mu_+(g)$}. The minimizer depends smoothly on the metric. 
From Lemma \ref{Wvariation}, we can show that $f_g$ satisfies the Euler-Lagrange equation\index{Euler-Lagrange equation}
\begin{align}\label{mueulerlagrange}-\Delta f_g-\frac{1}{2}|\nabla f_g|^2+\frac{1}{2}\scal_g -f_g=\mu_+(g).\end{align}
\begin{rem}Observe that $\mathcal{W}_+(\varphi^*g,\varphi^*f)=\mathcal{W}_+(g,f)$ for any diffeomorphism $\varphi$ and thus, $\mu_+(g)$ is invariant under diffeomorphisms.
\end{rem}
\begin{lem}[First variation of $\mu_+$]\label{mufirstvariation}\index{first variation!of $\mu_+$}The first variation of $\mu_+(g)$ is given by
\begin{align}\mu_+(g)'(h)=-\frac{1}{2}\int_M \langle \ric+g+\nabla^2 f_g,h\rangle e^{-f_g}\dv,
\end{align}
where $f_g$ realizes $\mu_+(g)$.
As a consequence, $\mu_+$ is nondecreasing under the Ricci flow \eqref{negativeflow}.
\end{lem}
\begin{proof}The first variational formula follows from Lemma \ref{Wvariation} and (\ref{mueulerlagrange}).
By diffeomorphism invariance\index{diffeomorphism!invariance},
\begin{align*}\mu_+(g)'(\nabla^2f_g)=\frac{1}{2}\mu_+'(g)(\mathcal{L}_{\gradient f_g}g)=0.
\end{align*}
Thus, if $g(t)$ is a solution of \eqref{negativeflow},
\begin{equation*}\frac{d}{dt}\mu_+(g(t))=\int_M  |\ric_{g(t)}+g(t)+\nabla^2 f_{g(t)}|^2 e^{-f_{g(t)}}\dv_{g(t)}\geq 0.
\qedhere
\end{equation*}
\end{proof}
\begin{rem}We call metrics gradient Ricci solitons if $\ric_g+\nabla^2 f=cg$ for some $f\in C^{\infty}(M)$ and $c\in\R$. 
In the compact case, any such metric is already Einstein if $c\leq0$ (see \cite[Proposition 1.1]{Cao10}). By the first variational formula of $\mu_+$\index{first variation!of $\mu_+$}, we conclude that
Einstein metrics with constant $-1$ are precisely the critical\index{critical} points of $\mu_+$.
\end{rem} 
\begin{lem}\label{f_{g_E}}Let $(M,g_E)$ be an Einstein manifold with constant $-1$. Furthermore, let $h\in \delta_{g_E}^{-1}(0)$. Then
\begin{itemize}
\item[(i)]$f_{g_E}\equiv \log\volume(M,g_E)$,
\item[(ii)]$\frac{d}{dt}|_{t=0}f_{g_E+th}=\frac{1}{2}\trace_{g_E}h$,
\item[(iii)]$\frac{d}{dt}|_{t=0}( \ric_{g_E+th}+g_E+th+\nabla_{g_E+th}^2 f_{g_E+th})=\frac{1}{2}\Delta_Eh,$
\end{itemize}
where $\Delta_E$ is the Einstein operator.
\end{lem}
\begin{proof}By substituting $w=e^{- f/2}$, we see that $w_{g_E}=e^{-f_{g_E}/2}$ is the minimizer of the functional
$$\widetilde{\mathcal{W}}(w)=\int_M 2|\nabla w|^2+\frac{1}{2}\scal w^2+w^2\log w^2 \dv $$
under the constraint
$$\left\|w\right\|_{L^2}=1.$$
By Jensen's inequality\index{Jensen's inequality}, we have a lower bound
\begin{align}\label{lowermu}\widetilde{\mathcal{W}}(w)\geq \frac{1}{2}\inf_{p\in M} \scal(p)-\log(\volume(M,g_E)),
\end{align}
which is realized by the constant function $w_{g_E}\equiv \volume(M,g_E)^{-1/2}$ since the scalar curvature is constant on $M$. This proves (i).
To prove (ii), we differentiate the Euler-Lagrange equation\index{Euler-Lagrange equation} \eqref{mueulerlagrange} in the direction of $h$. We obtain
\begin{align*}0=(-\Delta f)'-\frac{1}{2}(|\nabla f|^2)'+\frac{1}{2}\scal' -f'& =-(\Delta+1) f'+\frac{1}{2}(\Delta+1)\trace h.
\end{align*}
Here we used that $f_{g_E}$ is constant and $\delta h=0$.
The second assertion follows. It remains to show (iii). By straightforward differentiation,
\begin{align*}( \ric+g+\nabla^2 f)'= \frac{1}{2}\Delta_Lh-\frac{1}{2}\nabla^2\trace h+h+\frac{1}{2}\nabla^{2}\trace h= \frac{1}{2}\Delta_Eh.
\end{align*}
Here we used Lemma \ref{firstcuvature}, (i) and (ii).
\end{proof}
\begin{prop}[Second variation of $\mu_+$]\label{hessianmu}\index{second variation!of $\mu_+$}The second variation of $\mu_+$ at an Einstein metric satisfying $\ric_{g_E}=-g_E$ is given by
\begin{align*}\mu_+(g_E)''(h)=
\begin{cases} -\frac{1}{4}\fint_M\langle \Delta_E h,h\rangle \dv,& \text{ if }h\in \delta^{-1}(0),\\
              0,& \text{ if }h\in \delta^{*}(\Omega^1(M)),
\end{cases}
\end{align*}
where $\fint$ denotes the averaging integral\index{averaging integral}\index{$\fint$, averaging integral}, i.e.\ the integral divided by the volume.
\end{prop}
\begin{proof}Recall that the space of symmetric $(0,2)$-tensors splits as $\Gamma(S^2M)=\delta^{*}(\Omega^1(M))\oplus\delta^{-1}(0)$.
Since $\mu_+$ is a Riemannian functional, the Hessian restricted to $\delta^{*}(\Omega^1(M))$ vanishes. Now let $h\in\delta^{-1}(0)$. By the first variational formula\index{first variation!of $\mu_+$} and
Lemma \ref{f_{g_E}} (i) and (iii),
\begin{align*}\mu_+(g_E)''(h)=-\frac{1}{4}\fint_M\langle \Delta_E h,h\rangle \dv.
\end{align*}
Since $\delta(\Delta_E h)=\delta((\Delta_L+2 )h)=(\Delta_H+2)(\delta h)$ \cite[pp. 28-29]{Lic61}, $\Delta_E$ preserves $\delta^{-1}(0)$. 
Here, $\Delta_H$ is the Hodge-Laplacian acting on one-forms. 
Thus, the splitting $\delta^{*}(\Omega^1(M))\oplus \delta^{-1}(0)$ is orthogonal with respect to $\mu_+''$.
\end{proof}
\section{Some technical estimates}\label{technicalestimatesI}
In this section, we will establish bounds on $\mu_+$, $f_g$ and their variations in terms of certain norms of the variations. These estimates are needed in proving the main theorems of the next two sections.
\begin{lem}\label{goodf_g}Let $(M,g_E)$ be an Einstein manifold such that $\ric=-g_E$. Then there exists a $C^{2,\alpha}$-neighbourhood
$\mathcal{U}$ in the space of metrics such that
the minimizers $f_g$ are uniformly bounded in $C^{2,\alpha}$, i.e.\ there exists a constant $C>0$ such that $\left\| f_g\right\|_{C^{2,\alpha}}\leq C$ for all $g\in \mathcal{U}$.
Moreover, for each $\epsilon>0$, we can choose $\mathcal{U}$ so small that $\left\| \nabla f_g\right\|_{C^0}\leq \epsilon$ for all $g\in \mathcal{U}$.
\end{lem}
\begin{proof}
As in the proof of Lemma \ref{f_{g_E}} (i), we use the fact that  
\begin{align}\label{widetildeW}\mu_+(g)=\inf_{w\in C^{\infty}(M)}\widetilde{\mathcal{W}}(g,w)=\inf\int_M 2|\nabla w|^2+\frac{1}{2}\scal w^2+w^2\log w^2 \dv 
\end{align}
under the constraint
$$\left\|w\right\|_{L^2}=1.$$
There exists a unique minimizer\index{minimizer} of this functional which we denote by $w_g$. We have $w_g=e^{-f_g/2}$ and $w_g$ satisfies the Euler-Lagrange equation
\index{Euler-Lagrange equation}
\begin{align}\label{eulerlagrangew}2\Delta w_g+\frac{1}{2}\scal_g w_g-2w_g\log w_g=\mu_+(g)w_g.
\end{align}
We will now show that there exists a uniform bound $\left\|w_g\right\|_{C^{2,\alpha}}\leq C$ for all metrics $g$
in a $C^{2,\alpha}$-neighbourhood $\mathcal{U}$ of $g_E$.
For this purpose, we first remark that all the Sobolev constants that will appear below are uniformly bounded on $\mathcal{U}$.
 Now observe that by (\ref{widetildeW}),
$$2\left\|\nabla w_g\right\|_{L^2}\leq \mu_+(g)-C_1 \volume(M,g)-\frac{1}{2}\inf_{p\in M}\scal_g(p),$$
since the function $x\mapsto x\log x$ has a lower bound. 
By testing with suitable functions, one sees that $\mu_+(g)$ is bounded from above on $\mathcal{U}$. Therefore,
the $H^1$-norm of $\omega_g$ is bounded and by Sobolev embedding,
the same holds for the $L^{2n/(n-2)}$-norm. Let $p=2n/(n-2)$ and choose some $q$ slightly smaller than $p$.
By \eqref{eulerlagrangew} and elliptic regularity\index{elliptic regularity},
$$\left\|w_g\right\|_{W^{2,q}}\leq C_2 (\left\|w_g\log w_g\right\|_{L^q}+\left\|w_g\right\|_{L^q}).$$
Since $x\mapsto x\log x$ grows slower than $x\mapsto x^{\beta}$ for any $\beta>1$ as $x\to\infty$, we have the estimate
$$\left\|w_g\log w_g\right\|_{L^q}\leq C_3(\volume(M,g))+\left\|w_g\right\|_{L^p}.$$
This yields an uniform bound $\left\|w_g\right\|_{W^{2,q}}\leq C(q)$.

By Sobolev embedding\index{Sobolev!embedding}, we have uniform bounds on $\left\|w_g\right\|_{L^{p'}}$ for some $p'>p$ and by applying elliptic regularity on \eqref{eulerlagrangew},
we have bounds on $\left\|w_g\right\|_{W^{2,q'}}$ for every $q'<p'$.
Iterating this procedure, we obtain uniform bounds $\left\| w_g\right\|_{W^{2,p}}\leq C(p)$ for each $p\in (1,\infty)$. 
By choosing $p$ large enough, we can bound the $C^{0,\alpha}$-norm of $\omega_g$ and by elliptic regularity,
\index{elliptic regularity},
\begin{align*}\left\|w_g\right\|_{C^{2,\alpha}}&\leq C_4 (\left\|w_g\log w_g\right\|_{C^{0,\alpha}}+\left\|w_g\right\|_{C^{0,\alpha}})
                                              \leq C_5[ (\left\|w_g\right\|_{C^{0,\alpha}})^{\gamma}+\left\|w_g\right\|_{C^{0,\alpha}})\leq C_6
\end{align*}
Next, we show that the $C^{2,\alpha}$-norms of $f_g$ are uniformly bounded. First, we claim that we may choose a smaller neighbourhood $\mathcal{V}\subset\mathcal{U}$ such that 
for $g\in \mathcal{V}$, the functions $w_g$ are bounded away from zero (recall that any $w_g=e^{-f_g/2}$ is positive). Suppose this is not the case. Then there exists a sequence $g_i\to g_E$ in $C^{2,\alpha}$ such that $\min_p w_{g_i}(p)\to 0$ for $i\to\infty$.
Since $\left\|w_{g_i}\right\|_{C^{2,\alpha}}\leq C$ for all $i$, there exists a subsequence, again denoted by $w_{g_i}$ such that $w_{g_i}\to w_{\infty}$ in $C^{2,\alpha'}$ for some $\alpha'<\alpha$.
Obviously, the right hand side of (\ref{widetildeW}) converges.
Since $\mu_+$ is bounded from below by (\ref{lowermu}) and from above, a suitable choice of the subsequence ensures that also the left hand side of $(\ref{widetildeW})$ converges.
Therefore, $w_{\infty}$ equals the minimizer\index{minimizer} of $\widetilde{\mathcal{W}}(g_E,w)$, so $w_\infty=w_{g_E}=\volume(M,g_E)^{-1/2}$. In particular, $\min_p w_{g_i}(p)\to \volume(M,g_E)^{-1/2}\neq 0$ which
contradicts the assumption. 
Now we have
\begin{align*}\left\|f_g\right\|_{C^{2,\alpha}}=\left\|-2\log(w_g)\right\|_{C^{2,\alpha}}&\leq C(\log \inf w_g,1/(\inf w_g))\left\|w_g\right\|_{C^{2,\alpha}}\leq C_7.
\end{align*}
It remains to prove that for each $\epsilon>0$, we may choose $\mathcal{U}$ so small that $\left\|\nabla f_g\right\|_{C^0}< \epsilon$. We again use a subsequence
\index{subsequence} argument. Suppose this is not possible.
Then there exists a sequence\index{sequence} of metrics $g_i\to g$ in $C^{2,\alpha}$ and some $\epsilon_0>0$ such that for the corresponding $f_{g_i}$, the estimate $\left\|\nabla f_{g_i}\right\|_{C^0}\geq \epsilon_0$ holds for all $i$.
Because of the bound $\left\|f_g\right\|_{C^{2,\alpha}}\leq C$, we may choose a subsequence, again denoted by $f_i$ converging to some $f_{\infty}$ in $C^{2,\alpha'}$ for $\alpha'<\alpha$.
By the same arguments as above, $f_{\infty}=f_{g_E}\equiv-\log(\volume(M))$. In particular, $\left\|\nabla f_{g_i}\right\|_{C^0}\to 0$, a contradiction.
\end{proof}
\begin{lem}\label{f'}Let $(M,g_E)$ be an Einstein manifold such that $\ric_{g_E}=-g_E$. Then there exists a $C^{2,\alpha}$-neighbourhood
$\mathcal{U}$ of $g_E$ in the space of metrics and a constant $C>0$ such that for all $g\in \mathcal{U}$, we have
\begin{align*}
\left\|\frac{d}{dt}\bigg\vert_{t=0}f_{g+th}\right\|_{C^{2,\alpha}}\leq C  \left\|h\right\|_{C^{2,\alpha}},\qquad\left\|\frac{d}{dt}\bigg\vert_{t=0}f_{g+th}\right\|_{H^i}\leq C  \left\|h\right\|_{H^i},\quad i=1,2.
\end{align*}
\end{lem}
\begin{proof}Recall that $f_g$ satisfies the Euler-Lagrange equation\index{Euler-Lagrange equation}
$$-\Delta f-\frac{1}{2}|\nabla f|^2+\frac{1}{2}\scal -f=\mu_+(g).$$
Differentiating this equation in the direction of $h$ yields
$$-\dot{\Delta} f-\Delta\dot{f}+\frac{1}{2}h(\gradient f,\gradient f)-\langle \nabla f,\nabla \dot{f}\rangle+\frac{1}{2}\dot{\scal} -\dot{f}=\dot{\mu}_+(g).$$
By Lemma \ref{firsthessian} and Lemma \ref{firstcuvature} the variational formulas for the Laplacian and the scalar curvature
\index{first variation!of the Laplacian}\index{first variation!of the scalar curvature} are
\begin{align*}\dot{\Delta}f&=\langle h,\nabla^2 f\rangle-\langle \delta h+\frac{1}{2}\nabla \trace h,\nabla f\rangle,\\
              \dot{\scal}&=\Delta (\trace h)+\delta(\delta h)-\langle\ric,h\rangle.
\end{align*}
Because $\Delta+1$ is invertible, we can apply elliptic regularity\index{elliptic regularity} and we obtain
\begin{align*}\left\|\dot{f}\right\|_{C^{2,\alpha}}&\leq C_1\left\|(\Delta+1)\dot{f}\right\|_{C^{0,\alpha}}\\
&\leq C_1\left\|\nabla f\right\|_{C^0}\left\|\nabla\dot{f}\right\|_{C^{0,\alpha}}+C_1\left\|-\dot{\Delta} f+\frac{1}{2}h(\nabla f,\nabla f)+\frac{1}{2}\dot{\scal}-\dot{\mu}_+(g)\right\|_{C^{0,\alpha}}.
\end{align*}
By Lemma \ref{goodf_g}, we may choose $\mathcal{U}$ so small that $\left\|\nabla f\right\|_{C^{0}}<\epsilon$ for some small $\epsilon<\min\left\{C_1^{-1},1\right\}$. Then we have
\begin{align*}(1-\epsilon)\left\|\dot{f}\right\|_{C^{2,\alpha}}
&\leq C_1\left\|-\dot{\Delta} f+\frac{1}{2}h(\nabla f,\nabla f)+\frac{1}{2}\dot{\scal}-\dot{\mu}_+(g)\right\|_{C^{0,\alpha}}
\leq (C_2\left\| f_g\right\|_{C^{2,\alpha}}+C_3)\left\|h\right\|_{C^{2,\alpha}}.
\end{align*}
The last inequality follows from the variational formulas of the Laplacian, the scalar curvature and $\mu_+$. By the uniform bound on $\left\| f_g\right\|_{C^{2,\alpha}}$, the first estimate of the lemma follows.
The estimate of the $H^{i}$-norm is shown similarly.
\end{proof}
\begin{prop}[Estimate of the second variation of $\mu_+$]\label{secondmu}\index{second variation!of $\mu_+$}Let $(M,g_E)$ be an Einstein manifold with constant $-1$. Then there exists a $C^{2,\alpha}$-neighbourhood $\mathcal{U}$ of $g_E$
and a constant $C>0$ such that
$$\left|\frac{d^2}{dsdt}\bigg\vert_{s,t=0}\mu_+(g+th+sk)\right|\leq C \left\| h\right\|_{H^1}\left\|k\right\|_{H^1}$$
for all $g\in\mathcal{U}$.
\end{prop}
\begin{proof}By the formula of the first variation\index{first variation!of $\mu_+$},
\begin{align*}\frac{d^2}{dsdt}\bigg\vert_{s,t=0}\mu_+(g+th+sk)&=-\frac{d}{ds}\bigg\vert_{s=0}\frac{1}{2}\int_M \langle \ric_{g_s}+g_s-\nabla^2 f_{g_s},h\rangle_{g_s} e^{-f_{g_s}}\dv_{g_s}
                                                     =(1)+(2)+(3),
\end{align*}
and we estimate these three terms separately. The first term comes from differentiating the scalar product:
\begin{align*}|(1)|&=\left|\int_M \langle \ric_{g}+g-\nabla^2 f_{g},k\circ h\rangle_{g} e^{-f_{g}}\dv_{g}\right|
                   \leq C_1 \left\| h\right\|_{H^1}\left\|k\right\|_{H^1}.
\end{align*}
This estimates holds since the $f_g$ are uniformly bounded in a small $C^{2,\alpha}$-neighbourhood of $g_E$.
The second term comes from differentiating the gradient\index{gradient}:
\begin{align*}|(2)|&= \left|\frac{1}{2}\int_M \left\langle \frac{d}{ds}\bigg\vert_{s=0}(\ric_{g_s}+g_s-\nabla^2 f_{g_s}),h\right\rangle_{g} e^{-f_{g}}\dv_{g}\right|\\
                   &= \left|\frac{1}{2}\int_M \left\langle \frac{1}{2}\Delta_Lk-\delta^{*}(\delta k)-\frac{1}{2}\nabla^2\trace k+k-(\nabla^2)' f_g-\nabla^2 f'_g,e^{-f_{g}}h\right\rangle_{g} \dv_{g}\right|\\
                   &\leq C_2 \left\|k\right\|_{H^1}\left\|h\right\|_{H^1}.
\end{align*}
The inequality follows from integration by parts\index{integration by parts}, Lemma \ref{firsthessian}, Lemma \ref{f'} and from the uniform bound on the $f_g$.
The third term appears when we differentiate the measure:
\begin{align*}|(3)|&=\left|\frac{1}{2}\int_M \langle \ric_{g}+g-\nabla^2 f_{g},h\rangle_{g} \left(-f'_g+\frac{1}{2}\trace k\right)e^{-f_g}\dv_g\right|
                   \leq C_3  \left\| h\right\|_{H^1}\left\|k\right\|_{H^1}.
\end{align*}
Here we again used Lemma \ref{f'} in the last step.
\end{proof}
\begin{lem}\label{f''}Let $(M,g_E)$ be an Einstein manifold with constant $-1$. Then there exists a $C^{2,\alpha}$- neighbourhood $\mathcal{U}$ of $g_E$ and a constant $C>0$ such that
\begin{align*}\left\|\frac{d^2}{dtds}\bigg\vert_{t,s=0}f_{g+sk+th}\right\|_{H^i}\leq C \left\|h\right\|_{C^{2,\alpha}}\left\|k\right\|_{H^i},\qquad i=1,2.
\end{align*}
\end{lem}
 \begin{proof}In the proof, we denote $t$-derivatives by dot and $s$-derivatives by prime. Differentiating (\ref{mueulerlagrange}) twice yields
\begin{align*}-\Delta\dot{f}'-\dot{\Delta}f'-\Delta'\dot{f}-\dot{\Delta}'f+h(\gradient f,\gradient f')+k(\gradient f,\gradient\dot{f})&\\
                                            -\langle \nabla f,\nabla \dot{f}'\rangle-\langle\nabla\dot{f},\nabla f'\rangle+\frac{1}{2}\dot{\scal}'-\dot{f}'&=\dot{\mu}_+'.
\end{align*}
By elliptic regularity\index{elliptic regularity}, we have
\begin{align}\label{elliptic}\left\|\dot{f}'\right\|_{H^i}\leq C_1 \left\|(\Delta+1)\dot{f}'\right\|_{H^{i-2}}\leq C_1 \left\|\nabla f\right\|_{C^0}\left\|\nabla \dot{f}'\right\|_{L^2}+C_1\left\|(A)\right\|_{H^{i-2}},
\end{align}
where
\begin{align*}(A)=&-\dot{\Delta}f'-\Delta'\dot{f}-\dot{\Delta}'f+h(\gradient f,\gradient f')+k(\gradient f,\gradient\dot{f})
                                                -\langle \nabla \dot{f},\nabla f'\rangle+\frac{1}{2}\dot{\scal}'-\dot{\mu}_+'.
\end{align*}
By the first variation of the Laplacian and the scalar curvature and the estimates we already developed for $\dot{f}$ and $f'$ in Lemma \ref{f'}, we have
\begin{align*}
 \left\|-\dot{\Delta}f'-\Delta'\dot{f}+h(\gradient f,\gradient f')+k(\gradient f,\gradient\dot{f})-\langle \nabla \dot{f},\nabla f'\rangle
 \right\|_{L^2}\leq C_3 \left\|h\right\|_{C^{2,\alpha}}\left\|k\right\|_{H^1}.
\end{align*}
Now we consider the occurent second variational formulas of the Laplacian and the scalar curvature. 
\index{second variation!of the Laplacian}\index{second variation!of the scalar curvature}
By Lemma \ref{secondhessian}, they can be schematically written as
\begin{align*}\dot{\Delta}'f&=\nabla k*h*\nabla f+k*\nabla h*\nabla f,\\
\dot{\scal}'&=\nabla^2 k*h+k*\nabla^2 h+\nabla k*\nabla h+ R*k*h.
\end{align*}
Here, $*$\index{$*$, Hamilton's notation} is Hamilton's notation for a combination of tensor products\index{tensor product} with contractions\index{contraction}. Now, Lemma \ref{secondmu}, 
integration by parts\index{integration by parts} and the H\"older inequality\index{H\"older inequality} yield
\begin{align*}\left\|-\dot{\Delta}'f+\frac{1}{2}\dot{\scal}'-\dot{\mu}_+'\right\|_{H^{i-2}}\leq C_2 \left\|h\right\|_{C^{2,\alpha}}\left\|k\right\|_{H^i}.
\end{align*}
We obtain
\begin{align*}\left\|(A)\right\|_{H^{i-2}}\leq C_4\left\|h\right\|_{C^{2,\alpha}}\left\|k\right\|_{H^i}.
\end{align*}
Since $\left\|\nabla f\right\|_{C^0}$ can be assumed to be arbitrarily small, we bring this term to the left hand side of (\ref{elliptic}) and obtain the result.                       
\end{proof}
\begin{prop}[Estimates of the third variation of $\mu_+$]\label{thirdmu}\index{third variation!of $\mu_+$}Let $(M,g_E)$ be an Einstein manifold with constant $-1$. Then there exists a $C^{2,\alpha}$-neighbourhood $\mathcal{U}$ of $g_E$ 
and a constant $C>0$ such that
$$\left|\frac{d^3}{dt^3}\bigg\vert_{t=0}\mu_+(g+th)\right|\leq C \left\| h\right\|^2_{H^1}\left\|h\right\|_{C^{2,\alpha}}$$
for all $g\in\mathcal{U}$.
\end{prop}
\begin{proof}We have, by the first variational formula\index{first variation!of $\mu_+$},
\begin{align*}&\frac{d^3}{dt^3}\bigg\vert_{t=0}\mu_+(g+th)=-\frac{1}{2}\frac{d^2}{dt^2}\bigg\vert_{t=0}\int_M \langle \ric+g+\nabla^2 f_g,h\rangle e^{-f}\dv\\
                             &=-\frac{1}{2}\int_M \langle (\ric+g+\nabla^2 f_g)'',h\rangle e^{-f}\dv-3\int_M \langle \ric+g+\nabla^2 f_g,h\circ h\circ h\rangle e^{-f}\dv\\
                             &-\frac{1}{2}\int_M \langle \ric+g+\nabla^2 f_g,h\rangle (e^{-f}\dv)''+2\int_M \langle (\ric+g+\nabla^2 f_g)',h\circ h \rangle e^{-f}\dv\\
                              &-\int_M \langle (\ric+g+\nabla^2 f_g)',h\rangle (e^{-f}\dv)'+2\int_M \langle \ric+g+\nabla^2 f_g,h\circ h\rangle (e^{-f}\dv)'.
\end{align*}
Let us deal with the first term which contains
the second derivative of the gradient\index{gradient} of $\mu_+$. We have the schematic expressios
\begin{align*}(\ric+g)''&=\nabla^2h*h+\nabla h*\nabla h+R*h*h,\\
           (\nabla^2 f_g)''&=(\nabla^2)''f_g+2(\nabla^2)'f_g'+\nabla^2 f_g''\\
                           &=\nabla f*\nabla h*h+\nabla f'*\nabla h+\nabla^2 f_g'',
\end{align*}
see Lemma \ref{firsthessian} and Lemma \ref{secondhessian}.
From these expressions we obtain, by applying Lemma \ref{f'}, Lemma \ref{f''} and the H\"older inequality,
\begin{align*}\left|\int_M \langle (\ric+g+\nabla^2 f_g)'',h\rangle e^{-f}\dv\right|\leq C \left\| h\right\|^2_{H^1}\left\|h\right\|_{C^{2,\alpha}}.
\end{align*}
The estimates of the other terms are straightforward from the variational formulas in the appendix, Lemma \ref{f'} and Lemma \ref{f''}.
\end{proof}

\section{Local maximum of $\lambda$ and the expander entropy}\label{localmax}
Here we give characterizations of local maximality of $\lambda$ and $\mu_+$.
We prove Theorem \ref{lambdaandyamabe} using the theory developed for the Yamabe problem.
 For $\mu_+$, we use Koiso's local decomposition theorem of the space of metrics \cite{Koi79b} and the observation
that the $\mu_+$-functional can be explicitly evaluated on metrics of constant scalar curvature\index{constant scalar curvature}.
\begin{proof}[Proof of Theorem \ref{lambdaandyamabe}]
Suppose that $g_{RF}$ is not a local maximum of the Yamabe invariant. Then there exists a metric $g$ close to $g_{RF}$ such that $Y(M,[g])>0$ and by the solution of the Yamabe problem \cite{Sch84}, it admits a positive scalar curvature metric $\tilde{g}$ realizing $Y(M,[g])$ which is also close to $g_{RF}$ by  \cite[Theorem 2.5]{Koi79b}. Then by definition, $\lambda(\tilde{g})>0$, i.e.\ $g_{RF}$ is not a local maximizer of $\lambda$.

 Conversely, suppose that $g_{RF}$ is a local maximum of the Yamabe functional, i.e.\ the Yamabe invariant of any conformal class close to $[g_{RF}]$ is nonpositive. By the solution of the Yamabe problem, the sign of the smallest eigenvalue of the Yamabe operator $\Delta_Y=4\frac{n-1}{n-2}\Delta+\scal$ determines the sign of the Yamabe invariant. Thus, the smallest eigenvalue of $\Delta_Y$ is nonpositive on any metric close to $g_{RF}$. Since $\lambda$ is the smallest eigenvalue of the operator $4\Delta+\scal$ we nessecarily have $\lambda\leq 0$ for these metrics.
\end{proof}
\begin{thm}\label{yamabemaximality}Let $(M,g_E)$ be a compact Einstein manifold with constant $-1$. Then $g_E$ is a maximum of the $\mu_+$-functional in a $C^{2,\alpha}$-neighbourhood if and only if $g$ is a local maximum of the 
Yamabe functional\index{Yamabe!functional} in a $C^{2,\alpha}$-neighbourhood.
In this case, any metric sufficiently close to $g_E$ with $\mu_+(g)=\mu_+(g_E)$ is Einstein with constant $-1$.
\end{thm}
\begin{proof}Let $c=\volume(M,g_E)$ and write
\begin{align*}\mathcal{C}&=\left\{g\in\mathcal{M}|\scal_g \text{ is constant}\right\},\\
              \mathcal{C}_{c}&=\left\{g\in\mathcal{M}|\scal_g\text{ is constant and }\volume(M,g)=c\right\}.
\end{align*}
Since $\frac{\scal_{g_E}}{n-1}\notin \spectrum_+(\Delta_{g_E})$, \cite[Theorem 2.5]{Koi79b} asserts that the map
\begin{align*}\Phi\colon C^{\infty}(M)\times \mathcal{C}_{c}&\to\mathcal{M},\\
               (v,g)&\mapsto v\cdot g,
\end{align*}
is a local ILH-diffeomorphism\index{ILH!diffeomorphism} around $(1,g_E)$. Recall also that by \cite[Theorem C]{BWZ04}, any metric $g\in\mathcal{C}$ sufficiently close to $g_E$ is a Yamabe metric\index{Yamabe!metric}.

By the proof of Lemma \ref{f_{g_E}} (i), the minimizer\index{minimizer} $f_{\bar{g}}$ realizing $\mu_+(\bar{g})$ is constant if $\bar{g}\in\mathcal{C}$ and by the constraint in the definition, it equals $\log(\volume(M,\bar{g}))$. Thus, $\mu_+(\bar{g})=\frac{1}{2}\scal_{\bar{g}}-\log(\volume(M,\bar{g}))$.
If $g_E$ is not a local maximum of the Yamabe functional\index{Yamabe!functional}, there exist metrics $g_i\in\mathcal{C}_{c}$, $g_i\to g_E$ in $C^{2,\alpha}$ which have the same volume but larger scalar curvature than $g_E$.
Thus, also $\mu_+(g_i)>\mu_+(g_E)$ which causes the contradiction.

If $g_E$ is a local maximum of the Yamabe functional\index{Yamabe!functional}, it is a local maximum of $\mu_+$ restricted to $\mathcal{C}_{c}$. Any other metric $\bar{g}\in\mathcal{C}_{c}$
satisfying $\mu_+(\bar{g})=\mu_+(g_E)$ is also a local maximum of the Yamabe functional\index{Yamabe!functional}. In particular, $\bar{g}$ is a critical\index{critical} point of the total scalar curvature\index{total scalar curvature} restricted to $\mathcal{C}_{c}$ and
the scalar curvature is equal to $-n$. By Proposition \cite[Proposition 4.47]{Bes08}, $\bar{g}$ is an Einstein manifold with constant $-1$.
For $\alpha\cdot \bar{g}$, where $\alpha>0$ and $\bar{g}\in\mathcal{C}_{c}$ sufficiently close
to $g_E$, we have
\begin{align*}\mu_+(\alpha\cdot \bar{g})
                                        &=\frac{1}{2\alpha}\scal_{\bar{g}}-\frac{n}{2}\log(\alpha)-\log(\volume(M,\bar{g}))
                                        \leq-\frac{n}{2}-\log(\volume(M,g_E))=\mu_+(g_E),
\end{align*}
which shows that $g_E$ is also a local maximum of $\mu_+$ restricted to $\mathcal{C}$ and equality occurs if and only if $\alpha=1$ and  $\mu_+(\bar{g})=\mu_+(g_E)$.

It remains to investigate the variation of $\mu_+$ in the direction of volume-preserving\index{volume!preserving} conformal\index{conformal} deformations.
Let $h=v\cdot \bar{g}$, where $\bar{g}\in\mathcal{C}$ and $v\in C^{\infty}(M)$ with $\int_Mv \dv_{\bar{g}}=0$. Then
\begin{align*}\frac{d}{dt}\bigg\vert_{t=0}\mu_+(\bar{g}+th)&=-\frac{1}{2}\int_M\langle \ric_{\bar{g}}+\bar{g},h\rangle e^{-f_{\bar{g}}}\dv
                                            =-\frac{1}{2}\fint_M (\scal_{\bar{g}}+n)v \dv=0,
\end{align*}
since $f_{\bar{g}}$ is constant.
The second variation equals
\begin{align*}\frac{d^2}{dt^2}\bigg\vert_{t=0}\mu_+(\bar{g}+th)
                                                 =&-\frac{1}{2}\fint_M \left\langle \frac{d}{dt}\bigg\vert_{t=0}(\ric_{\bar{g}+th}+\bar{g}+th+\nabla^2 f_{\bar{g}+th}), h\right\rangle_{\bar{g}} \dv_{\bar{g}}\\
                                                  &+\fint_M \langle \ric_{\bar{g}}+\bar{g},h\circ h\rangle_{\bar{g}} \dv_{\bar{g}}
                                                  -\frac{1}{2}\fint_M \langle \ric_{\bar{g}}+\bar{g},h\rangle \left(-f'+\frac{1}{2}\trace h\right)\dv_{\bar{g}}.
\end{align*}
By the first variation of the Ricci tensor\index{first variation!of the Ricci tensor},
\begin{align*}-\frac{1}{2}\fint_M \langle \ric'&+h, h\rangle  \dv_{\bar{g}}
=-\frac{n-1}{2}\fint_M |\nabla v|^2\dv_{\bar{g}}-\frac{n}{2}\fint_Mv^2 \dv_{\bar{g}}.
\end{align*}
By differentiating Euler-Lagrange equation\index{Euler-Lagrange equation} \eqref{mueulerlagrange}, we have
\begin{align}\label{conformaleulerlagrange}(\Delta+1)f'=\frac{1}{2}((n-1)\Delta v-\scal_{\bar{g}}v).
\end{align}
Thus,
\begin{align*}-\frac{1}{2}\int_M&\left\langle \frac{d}{dt}\bigg\vert_{t=0}\nabla^2 f_{\bar{g}+th},h\right\rangle e^{-f_{\bar{g}}}\dv
                                                      =\frac{1}{4}\fint_M [(n-1)\Delta v-\scal_{\bar{g}}v]v \dv-\frac{1}{2}\fint_M f'\cdot v \dv.
\end{align*}
Adding up, we obtain
\begin{align*}-\frac{1}{2}\int_M &\left\langle \frac{d}{dt}\bigg\vert_{t=0}(\ric_{\bar{g}+th}+\bar{g}+th+\nabla^2 f_{\bar{g}+th}), h\right\rangle_{\bar{g}}e^{-f_{\bar{g}}} \dv_{\bar{g}}\\
              =&-\frac{1}{4}\fint_M |\nabla v|^2\dv_{\bar{g}}-\frac{1}{2}\left(n+\frac{\scal_{\bar{g}}}{2}\right)\fint_Mv^2\dv-\frac{1}{2}\fint_M f'\cdot v  \dv
             \leq-C_1\left\|v\right\|_{H^1}^2,
\end{align*}
and this estimate is uniform in a small $C^{2,\alpha}$-neighbourhood of $g_E$. 
Here we have used that by (\ref{conformaleulerlagrange}), the $L^2$-scalar product of $f'$ and $v$ is positive.
Given any $\epsilon>0$, the remaining terms of the second variation can be estimated by
\begin{align*}\int_M \langle \ric_{\bar{g}}+\bar{g},h\circ h\rangle_{\bar{g}}e^{-f_{\bar{g}}} \dv_{\bar{g}}=(\scal_{\bar{g}}+n)\fint_M v^2 \dv\leq \epsilon \left\|v\right\|^2_{L^2}
\end{align*}
and
\begin{align*}
-\frac{1}{2}\int_M \langle \ric_{\bar{g}}+\bar{g},h\rangle& \left(-f'_{\bar{g}}+\frac{1}{2}\trace h\right)e^{-f_{\bar{g}}}\dv_{\bar{g}}
=-\frac{\scal_{\bar{g}}+n}{2}\fint_M v\left(-f'_{\bar{g}}+\frac{n}{2}v\right)\dv
 \leq \epsilon \left\|v\right\|_{L^2}^2,
\end{align*}
provided that the neighbourhood is small enough. In the last inequality, we used $\left\|f'\right\|_{L^2}\leq C_2\left\|v\right\|_{L^2}$ which holds because of (\ref{conformaleulerlagrange}) and elliptic regularity\index{elliptic regularity}. Thus, we have a uniform estimate
\begin{align*}\frac{d^2}{dt^2}\bigg\vert_{t=0}\mu_+(\bar{g}+tv\bar{g})\leq -C_3\left\|v\right\|_{H^1}^2.
\end{align*}
Let now $g$ be an arbitrary metric in a small $C^{2,\alpha}$-neighbourhood of $g_E$. By the above, it can be written as $g=\tilde{v}\cdot\tilde{g}$, where $(\tilde{v},\tilde{g})\in C^{\infty}(M)\times\mathcal{C}_{g_E}$ is close to $(1,g_E)$.
By substituting 
\begin{align*}v=\frac{\tilde{v}-\fint \tilde{v}\dv_{\tilde{g}}}{\fint \tilde{v}\dv_{\tilde{g}}},\qquad\bar{g}=\left(\fint \tilde{v}\dv_{\tilde{g}}\right)\tilde{g},
\end{align*}
we can write $g=(1+v)\bar{g}$, where $\bar{g}\in\mathcal{C}$ is close to $g_E$ and $v\in C^{\infty}(M)$ with $\int_M v\dv_{\bar{g}}=0$ is close to $0$.
Thus by Taylor expansion\index{Taylor expansion} and Proposition \ref{thirdmu},
\begin{align*}\mu_+(g)&=\mu_+(\bar{g})+\frac{1}{2}\frac{d^2}{dt^2}\bigg\vert_{t=0}\mu_+(\bar{g}+tv\bar{g})+\int_0^1\left(\frac{1}{2}-t+\frac{1}{2}t^2\right)\frac{d^3}{dt^3}\mu_+(\bar{g}+tv\bar{g})dt\\
                        &\leq \mu_+(g_E)-C_4 \left\|v\right\|_{H^1}^2+ C_5\left\| v\right\|_{C^{2,\alpha}}\left\|v\right\|_{H^1}^2.
\end{align*}
 Now if we choose the $C^{2,\alpha}$-neighbourhood small enough, $\mu_+(g)\leq\mu_+(g_E)$ and equality holds if and only if $v\equiv0$ and $\mu_+(g)=\mu_+(g_E)$. As discussed earlier in the proof, this implies that $g$ is Einstein with constant $-1$.
\end{proof}
\section{A Lojasiewicz-Simon inequality}
For proving a gradient inequality for $\mu_+$, we need to know that $\mu_+$ is analytic. To show this, we use the implicit function\index{implicit function theorem} theorem for Banach manifolds\index{Banach manifold} in the analytic\index{analytic} category mentioned in \cite[Section 13]{Koi83}.
Such arguments were also used in \cite[Lemma 2.2]{SW13} which is a result similar to the below lemma.
\begin{lem}\label{analyticmu}There exists a $C^{2,\alpha}$-neighbourhood $\mathcal{U}$ of $g_E$ such that the map $g\mapsto \mu_+(g)$ is analytic on $\mathcal{U}$.
\end{lem}
\begin{proof}Let $H(g,f)=-\Delta_g f-\frac{1}{2}|\nabla f|^2+\frac{1}{2}\scal_g-f$ and consider the map\index{$\mathcal{M}^{C^{2,\alpha}}$, set of $C^{2,\alpha}$-metrics}
\begin{align*}L\colon\mathcal{M}^{C^{2,\alpha}}\times C^{2,\alpha}(M)&\to C_{g_E}^{0,\alpha}(M)\times\R,\\
               (g,f)&\mapsto \left(H(g,f)-\fint_M H(g,f)\dv_{g_E},\int_M e^{-f}\dv_g-1\right).
\end{align*}
Here, $\mathcal{M}^{C^{2,\alpha}}$, is the set of $C^{2,\alpha}$-metrics and $C_{g_E}^{k,\alpha}(M)=\left\{f\in C^{k,\alpha}(M)|\int_M f\dv_{g_E}=0\right\}$.
This is an analytic\index{analytic} map between Banach manifolds\index{Banach manifold}. Observe that $L(g,f)=(0,0)$ if and only if we have $H(f,g)=const$ and $\int_Me^{-f}\dv_g=1$. The differential of $L$ at $(g_E,f_{g_E})$
 restricted to its second argument is equal to 
\begin{align*}dL_{g_E,f_{g_E}}(0,v)=\left(-(\Delta_{g_E}+1)v+\fint_M v\dv,-\fint_M v \dv\right).
\end{align*}
The map $dL_{g_E,f_{g_E}}|_{C^{2,\alpha}(M)}:C^{2,\alpha}(M)\to C_{g_E}^{0,\alpha}(M)\times\R$ is a linear isomorphism\index{isomorphism}
because it acts as $-(\Delta_{g_E}+1)$ on $C^{2,\alpha}_{g_E}$ and as $-\identity$ on constant functions.
By the implicit function theorem\index{implicit function theorem} for Banach manifolds\index{Banach manifold}, there exists a
neighbourhood $\mathcal{U}\subset\mathcal{M}^{C^{2,\alpha}}$ and an analytic map $P\colon\mathcal{U}\to C^{2,\alpha}(M)$ such that we have $L(g,P(g))=(0,0)$. Moreover, there exists a neighbourhood
$\mathcal{V}\subset C^{2,\alpha}(M)$ of $f_{g_E}$ such that if $L(g,f)=0$ for some $g\in \mathcal{U},f\in\mathcal{V}$, then $f=P(g)$.

Next, we show that $f_{g}=P(g)$ for all $g\in\mathcal{U}$ (or eventually on a smaller neighbourhood). Suppose this is not the case. Then there exists a sequence\index{sequence} $g_i$ which converges to $g$ in $C^{2,\alpha}$
and such that $f_i\neq P(g_i)$ for all $i$. By the proof of Lemma \ref{goodf_g}, $\left\|f_{g_i}\right\|_{C^{2,\alpha}}$ is bounded and for every $\alpha'<\alpha$, there is a subsequence\index{subsequence}, again denoted by $f_{g_i}$ converging to $f_{g_E}$
in $C^{2,\alpha'}$. We obviously have $L(g_i,f_{g_i})=(0,0)$ and for sufficiently large $i$ we have, by the implicit function theorem\index{implicit function theorem}, $f_{g_i}=P(g_i)$. This causes the contradiction.

We immediately get that $\mu_+(g)=H(g,P(g))$ is analytic\index{analytic} on $\mathcal{U}$ since $H$ and $P$ are analytic.
\end{proof}
\begin{thm}[Lojasiewicz-Simon inequality for $\mu_+$]\label{nonintegrablegradient}\index{Lojasiewicz-Simon inequality}Let $(M,g_E)$ be a Einstein manifold with constant $-1$. Then there exists a $C^{2,\alpha}$-neighbourhood $\mathcal{U}$
of $g_E$ in the space of metrics and constants $\sigma\in[1/2,1)$, $C>0$ such that
\begin{align}\label{gradientmu2}|\mu_+(g)-\mu_+(g_E)|^{\sigma}\leq C \left\|\ric_g+g+\nabla^2 f_g\right\|_{L^2}
\end{align}
for all $g\in\mathcal{U}$.
\end{thm}
\begin{proof}The proof is an application of a general Lojasiewicz-Simon inequality\index{Lojasiewicz-Simon inequality} which was proven in \cite{CM14}. Here the analyticity\index{analytic} of $\mu_+$ is crucial.

Since both sides are diffeomorphism invariant\index{diffeomorphism!invariance}, it suffices to show the inequality on a slice\index{slice} to the action 
of the diffeomorphism group.\index{diffeomorphism!group} Let
\begin{align*}\mathcal{S}_{g_E}=\mathcal{U}\cap\left\{g_E+h\suchthat{h\in \delta^{-1}_{g_E}(0)}h\in \delta^{-1}_{g_E}(0)\right\},
\end{align*}
and let $\tilde{\mu}_+$ be the $\mu_+$-functional restricted to $\mathcal{S}_{g_E}$. Obviously, $\tilde{\mu}_+$ is analytic\index{analytic} since $\mu_+$ is. The $L^2$-gradient\index{gradient!$L^2$} of $\mu_+$ is given by $\nabla\mu_+(g)=-\frac{1}{2}(\ric_g+g+\nabla^2 f_g)e^{-f_g}$. 
It vanishes at $g_E$. On the neighbourhood $\mathcal{U}$, we have the uniform estimate
\begin{align}\label{H^2estimateformu}\left\|\nabla\mu_+(g_1)-\nabla\mu_+(g_2)\right\|_{L^2}\leq C_1\left\|g_1-g_2\right\|_{H^2},
\end{align}
which holds by Taylor expansion\index{Taylor expansion} and Lemma \ref{f'}. The $L^2$-gradient\index{gradient!$L^2$} of $\tilde{\mu}_+$ is given by the projection\index{projection} of $\nabla\mu_+$ to $\delta^{-1}_{g_E}(0)$. Therefore, (\ref{H^2estimateformu}) also holds for $\nabla\tilde{\mu}_+$.
The linearization\index{linearization} of $\tilde{\mu}_+$ at $g_E$ is (up to a constant factor) given by the Einstein operator, see Lemma \ref{f_{g_E}} (iii). By ellipticity\index{ellipticity},
\begin{align*}\Delta_{E}\colon(\delta^{-1}_{g_E}(0))^{C^{2,\alpha}}\to(\delta^{-1}_{g_E}(0))^{C^{0,\alpha}}
\end{align*}\index{$(\delta^{-1}_{g_E}(0))^{C^{k,\alpha}}$, space of $C^{k,\alpha}$ divergence-free tensors }
is Fredholm. It also satisfies the estimate $\left\|\Delta_{E}h\right\|_{L^2}\leq C _2\left\|h\right\|_{H^2}$. 

By \cite[Theorem 7.3]{CM14}, there exists a constant $\sigma\in [1/2,1)$ such that $|\mu_+(g)-\mu_+(g_E)|^{\sigma}\leq \left\|\nabla\tilde{\mu}_+(g)\right\|_{L^2}$
for any $g\in\mathcal{S}_{g_E}$. Since
\begin{align*}\left\|\nabla\tilde{\mu}_+(g)\right\|_{L^2}\leq \left\|\nabla\mu_+(g)\right\|_{L^2}\leq  C_3 \left\|\ric_g+h+\nabla^2 f_g\right\|_{L^2},
\end{align*}
(\ref{gradientmu2}) holds for all $g\in\mathcal{S}_{g_E}$. By the slice theorem (\cite[Theorem 7.1]{Eb70}), any metric in $\mathcal{U}$ is isometric to some metric in $\mathcal{S}_{g_E}$.
Thus by diffeomorphism invariance\index{diffeomorphism!invariance}, \eqref{gradientmu2} holds for all $g\in \mathcal{U}$.
\end{proof}
\section{Dynamical stability and instability}\label{negativescalarcurvture}
With the characterization of the maximality of $\mu_+$  and the Lojasiewicz-Simon inequality, we are nearly ready to prove the dynamical stability and instablity theorems in the case of negative scalar curvature.
In this section, a Ricci flow is always of the form \eqref{negativeflow}.
Two preparing lemmas are left:
\begin{lem}[Estimates for $t\leq 1$]\label{shorttimeestimates}Let $(M,g_E)$ be an Einstein manifold with constant $-1$ and let $k\geq2$. Then for all $\epsilon>0$ there exists a $\delta>0$ such that if 
$\left\|g_0-g_E\right\|_{C^{k+2}_{g_E}}<\delta$, the Ricci flow starting at $g_0$ exists on $[0,1]$ and satisfies
\begin{equation*}\left\|g(t)-g_E\right\|_{C^k_{g_E}}<\epsilon
\end{equation*}
for all $t\in [0,1]$.
\end{lem}
\begin{proof}This follows from the evolution inequalites of the Riemann and the Ricci tensor under the Ricci flow \eqref{negativeflow} and the maximum principle for scalars exactly as in \cite[Lemma 5.1]{Has12}.
 \end{proof}
\begin{lem}\label{nablacurvature}Let $g(t)$, $t\in [0,T]$ be a solution of the Ricci flow and suppose that
\begin{align*}\sup_{p\in M}|R_{g(t)}|_{g(t)}\leq T^{-1}\qquad \forall t\in[0,T].
\end{align*}
Then for each $k\geq 1$, there exists a constant $C(k)$ such that
\begin{align*} \sup_{p\in M}|\nabla^kR_{g(t)}|_{g(t)}\leq C(k)\cdot T^{-1} t^{-k/2}\qquad\forall t\in(0,T].
\end{align*}
\end{lem}
\begin{proof}
 This is a well known result for the standard Ricci flow \cite[Theorem 7.1]{Ham95}. The proof also works for the flow \eqref{negativeflow},
 because the evolution inequality of the Riemann tensor needed in the proof is also satisfied under the flow \eqref{negativeflow}.
\end{proof}
\begin{thm}[Dynamical stability]\label{negativemodulostability}
\index{stable!dynamically}\index{modulo diffeomorphism}Let $(M,g_E)$ be an Einstein manifold with constant $-1$. Let $k\geq3$. If $g_E$ is a local maximizer of the Yamabe functional\index{Yamabe!functional}, then for every $C^{k}$-neighbourhood $\mathcal{U}$
of $g_E$, there exists a $C^{k+2}$-neighbourhood $\mathcal{V}$ such that the following holds:

For any metric $g_0\in\mathcal{V}$ there exists a $1$-parameter family\index{1@$1$-parameter family} of diffeomorphisms $\varphi_t$ such that for the Ricci flow $g(t)$ solving \eqref{negativeflow} which starts at $g_0$,
the modified flow $\varphi_t^*g(t)$ stays in $\mathcal{U}$ for all time and converges to an Einstein metric $g_{\infty}$ with constant $-1$ in $\mathcal{U}$ as $t\to\infty$.
The convergence is of polynomial rate, i.e.\ there exist constants $C,\alpha>0$ such that
\begin{align*}\left\|\varphi_t^*g(t)-g_{\infty}\right\|_{C^k}\leq C(t+1)^{-\alpha}.
\end{align*}
\end{thm}
\begin{proof}
We write $\mathcal{B}^k_{\epsilon}$\index{$\mathcal{B}^k_{\epsilon}$, $\epsilon$-ball w.r.t.\ the $C^k_{g_E}$-norm} for the $\epsilon$-ball around $g_E$ with respect to the $C^k_{g_E}$-norm.
Without loss of generality, we may assume that $\mathcal{U}=\mathcal{B}^k_{\epsilon}$
and $\epsilon>0$ is so small that Theorems \ref{yamabemaximality} and \ref{nonintegrablegradient} hold on $\mathcal{U}$.

By Lemma \ref{shorttimeestimates}, we can choose a small neighbourhood $\mathcal{V}$ such that the Ricci flow starting at any metric $g\in\mathcal{V}$ stays in $\mathcal{B}^k_{\epsilon/4}$
up to time $1$. Let $T\geq 1$ be the maximal time such that for any Ricci flow $g(t)$ starting in $\mathcal{V}$,
there exists a family of diffeomorphisms $\varphi_t$ such that the modified flow $\varphi_t^*g(t)$ stays in $\mathcal{U}$. By definition of $T$ and diffeomorphism invariance\index{diffeomorphism!invariance}, we have uniform curvature bounds
\begin{align*}\sup_{p\in M}|R_{g(t)}|_{g(t)}\leq C_1\qquad \forall t\in[0,T).
\end{align*}
By Lemma \ref{nablacurvature}, we have
\begin{align}\label{nablal_R}\sup_{p\in M}|\nabla^l R_{g(t)}|_{g(t)}\leq C(l)\qquad \forall t\in[1,T).
\end{align}
Because $f_{g(t)}$ satisfies the equation $-\Delta f_g-\frac{1}{2}|\nabla f_g|^2+\frac{1}{2}\scal_g -f_g=\mu_+(g)$, we also have
\begin{align}\label{nablal_f}\sup_{p\in M}|\nabla^l f_{g(t)}|_{g(t)}\leq \tilde{C}(l),\qquad \forall t\in[1,T).
\end{align}
Note that all these estimates are diffeomorphism invariant.

We now construct a modified Ricci flow as follows: Let $\varphi_t\in \Diff(M)$, $t\geq 1$ be the family of diffeomorphisms generated by
$X(t)=-\gradient_{g(t)}f_{g(t)}$ and define
\begin{align}\label{modifiedflow}\tilde{g}(t)=\begin{cases}g(t), & t\in[0,1],\\
                                   \varphi_t^*g(t),  & t\geq 1.
\end{cases}
\end{align}
The modified flow satisfies \eqref{negativeflow} for $t\in[0,1]$ while for $t\geq 1$, we have
\begin{align*}\frac{d}{dt}\tilde{g}(t)&=-2(\ric_{\tilde{g}(t)}+\tilde{g}(t)+\nabla^2 f_{\tilde{g}(t)}).
\end{align*}
 Let $T'\in[0,T]$ be the maximal time such that the modified Ricci flow, starting at any metric $g_0\in\mathcal{V}$, stays in $\mathcal{U}$ up to time $T'$.
 Then
\begin{align*}\left\|\tilde{g}(T')-g_E\right\|_{C^k_{g_E}}\leq& \left\|\tilde{g}(1)-g_E\right\|_{C^k_{g_E}}+\int_1^{T'}\left\| \dot{\tilde{g}}(t)\right\|_{C^k_{g_E}}dt
                                                               \leq\frac{\epsilon}{4}+2\int_1^{T'}\left\| \dot{\tilde{g}}(t)\right\|_{C^{k}_{\tilde{g}(t)}}dt,
\end{align*}
provided that $\mathcal{U}$ is small enough.
By the interpolation\index{interpolation} inequality for tensors (see \cite[Corollary 12.7]{Ham82}), (\ref{nablal_R}) and (\ref{nablal_f}), we have 
\begin{align*}\left\| \dot{\tilde{g}}(t)\right\|_{C^k_{\tilde{g}(t)}}\leq C_2\left\| \dot{\tilde{g}}(t)\right\|^{1-\eta}_{L^2_{\tilde{g}(t)}}
\end{align*}
for $\eta$ as small as we want. In particular, we can assume that $\theta:=1-\sigma(1+\eta)>0,$ where $\sigma$ is the exponent appearing in Theorem \ref{nonintegrablegradient}.
By the first variation of $\mu_+$\index{first variation!of $\mu_+$},
\begin{align*}\frac{d}{dt}\mu_+(\tilde{g}(t))\geq C_3\left\| \dot{\tilde{g}}(t)\right\|^{1+\eta}_{L^2_{\tilde{g}(t)}}\left\| \dot{\tilde{g}}(t)\right\|^{1-\eta}_{L^2_{\tilde{g}(t)}}.
\end{align*}
By Theorem \ref{yamabemaximality} and Theorem \ref{nonintegrablegradient} again,
\begin{align*}-\frac{d}{dt}|\mu_+(\tilde{g}(t))-&\mu_+(g_E)|^{\theta}=\theta|\mu_+(\tilde{g}(t))-\mu_+(g_E)|^{\theta-1}\frac{d}{dt}\mu_+(\tilde{g}(t))\\
                                                                    &\geq C_4|\mu_+(\tilde{g}(t))-\mu_+(g_E)|^{-\sigma(1+\eta)}\left\| \dot{\tilde{g}}(t)\right\|^{1+\eta}_{L^2_{\tilde{g}(t)}}
                                                                    \left\| \dot{\tilde{g}}(t)\right\|^{1-\eta}_{L^2_{\tilde{g}(t)}}
                                                       \geq C_5\left\| \dot{\tilde{g}}(t)\right\|_{C^k_{\tilde{g}(t)}}.
\end{align*}
 Hence by integration,
  \begin{align*}
   \int_1^{T'}\left\| \dot{\tilde{g}}(t)\right\|_{C^k_{\tilde{g}(t)}}dt
\leq \frac{1}{C_5}|\mu_+(\tilde{g}(1))-\mu_+(g_E)|^{\theta}\leq \frac{1}{C_5}|\mu_+(\tilde{g}(0))-\mu_+(g_E)|^{\theta}\leq\frac{\epsilon}{8},
  \end{align*}
 provided that $\mathcal{V}$ is small enough. This shows that $\left\|\tilde{g}(T')-g_E\right\|_{C^k_{g_E}}\leq \epsilon/2<\epsilon$, so $T'$ cannot be finite.
 Thus, $T=\infty$ and $\tilde{g}(t)$ converges to some limit metric $g_{\infty}\in\mathcal{U}$ as $t\to\infty$. By the Lojasiewicz-Simon inequality\index{Lojasiewicz-Simon inequality}, we have
\begin{align*}\frac{d}{dt}|\mu_+(\tilde{g}(t))-\mu_+(g_E)|^{1-2\sigma}\geq C_6,
\end{align*}
which implies
\begin{align*}|\mu_+(\tilde{g}(t))-\mu_+(g_E)|\leq C_7(t+1)^{-\frac{1}{2\sigma-1}}.
\end{align*}
Here, we may assume that $\sigma>\frac{1}{2}$ because the Lojasiewicz-Simon inequality\index{Lojasiewicz-Simon inequality} also holds after enlarging the exponent.
Therefore, $\mu_+(g_{\infty})=\mu_+(g_E)$, so $g_{\infty}$ is an Einstein metric with constant $-1$. The convergence is of polynomial rate since for $t_1<t_2$,
\begin{align*}\left\|\tilde{g}(t_1)-\tilde{g}(t_2)\right\|_{C^k}\leq C_8|\mu_+(\tilde{g}(t_1))-\mu_+(g_E)|^{\theta}\leq C_9(t_1+1)^{-\frac{\theta}{2\sigma-1}},
\end{align*}
and the assertion follows from $t_2\to\infty$.
\end{proof}
\begin{thm}[Dynamical instability]\label{negativemoduloinstability}
\index{unstable!dynamically}\index{modulo diffeomorphism}
Let $(M,g_E)$ be an Einstein manifold with constant $-1$ which is not a local maximizer of the Yamabe functional\index{Yamabe!functional}. Then there exists a nontrivial ancient\index{ancient} Ricci flow $g(t)$ solving \eqref{negativeflow}, defined on
 $(-\infty,0]$, and a $1$-parameter family of diffeomorphisms $\varphi_t$, $t\in (-\infty,0]$ such that $\varphi_t^*g(t)\to g_E$ as $t\to-\infty$.
\end{thm}
\begin{proof}Since $(M,g_E)$ is not a local maximum of the Yamabe functional, it cannot be a local maximum of $\mu_+$. Let $g_i\to g_E$ in $C^k$ and suppose that
we have $\mu_+(g_i)>\mu_+(g_E)$ for all $i$.
Let $\tilde{g}_i(t)$ be the modified flow defined in \eqref{modifiedflow}, starting at $g_i$.
Then by Lemma \ref{shorttimeestimates}, $\bar{g}_i=g_i(1)$ converges to $g_E$ in $C^{k-2}$ and by monotonicity\index{monotonicity}, $\mu_+(\bar{g}_i)>\mu_+(g_E)$ as well.
Let $\epsilon>0$ be so small that Theorem \ref{nonintegrablegradient} holds on $\mathcal{B}^{k-2}_{2\epsilon}$.
Theorem \ref{nonintegrablegradient} yields the differential inequality\index{differential inequality}
\begin{align*}\frac{d}{dt}(\mu_+(\tilde{g}_i(t))-\mu_+(g_E))^{1-2\sigma}\leq -C_1,
\end{align*}
from which we obtain
\begin{align*} \mu_+(\tilde{g}_i(t))\geq[(\mu_+(\tilde{g}_i(1))-\mu_+(g_E))^{1-2\sigma}-C_1(t-1)]^{-\frac{1}{2\sigma-1}}+\mu_+(g_E) 
\end{align*}
 as long as $\tilde{g}_i(t)$ stays in $\mathcal{B}^{k-2}_{2\epsilon}$.
Thus, there exists a $t_i$ such that
\begin{align*}\left\|\tilde{g}_i(t_i)-g_E\right\|_{C^{k-2}}=\epsilon,
\end{align*}
and $t_i\to \infty$. If $\left\{t_i\right\}$ was bounded, $\tilde{g}_i(t_i)\to g_E$ in $C^{k-2}$. By interpolation\index{interpolation},
\begin{align*}\left\|\ric_{\tilde{g}_i(t)}+\nabla^2f_{\tilde{g}_i(t)}+\tilde{g}_i(t)\right\|_{C^{k-2}}\leq C_2 \left\|\ric_{\tilde{g}_i(t)}+\nabla^2f_{\tilde{g}_i(t)}+\tilde{g}_i(t)\right\|_{L^2}^{1-\eta}
\end{align*}
for $\eta>0$ as small as we want. We may assume that $\theta=1-\sigma(1+\eta)>0$.
By Theorem \ref{nonintegrablegradient} , we have the differential inequality\index{differential inequality}
\begin{align*}\frac{d}{dt}(\mu_+(\tilde{g}_i(t))&-\mu_+(g_E))^{\theta}\geq C_3\left\|\ric_{\tilde{g}_i(t)}+\tilde{g}_i(t)\right\|_{L^2}^{1-\eta},
\end{align*}
if $\mu_+(\tilde{g}_i(t))>\mu_+(g_E)$. Thus,
\begin{align}\label{nontrivial3}\epsilon=\left\|\tilde{g}_i(t_i)-g_E\right\|_{C^{k-2}}\leq\left\|\bar{g}_i-g_E\right\|_{C^{k-2}}+C_4(\mu_+(\tilde{g}_i(t_i))-\mu_+(g_E))^{\theta}.
\end{align}
Now put $\tilde{g}_i^s(t):=\tilde{g}_i(t+t_i)$, $t\in[T_i,0]$, where $T_i=1-t_i\to-\infty$. We have
\begin{align*}\left\|\tilde{g}^s_i(t)-g_E\right\|_{C^{k-2}}&\leq\epsilon\qquad\forall t\in[T_i,0],\\
                \tilde{g}_i^s(T_i)&\to g_E\quad \text{ in }C^{k-2}.
\end{align*}
Because the embedding $C^{k-3}(M)\subset C^{k-2}(M)$ is compact\index{compact embedding}, we can choose a subsequence of the $\tilde{g}_i^s$, converging in $C^{k-3}_{loc}(M\times (-\infty,0])$ to an \index{ancient}ancient flow $\tilde{g}(t)$, $t\in (-\infty,0]$,
satisfying the differential equation\index{differential equation}
\begin{align*} \dot{\tilde{g}}(t)=-2(\ric_{\tilde{g}(t)}+\tilde{g}(t)+\nabla^2 f_{\tilde{g}(t)}).
\end{align*}
Let $\varphi_t$, $t\in(-\infty,0]$ be the diffeomorphisms generated by $X(t)=\grad_{\tilde{g}(t)}f_{\tilde{g}}$ where $\varphi_0=\mathrm{id}$. Then $g(t)=\varphi_t^*\tilde{g}(t)$
is a solution of \eqref{negativeflow}.
From taking the limit $i\to\infty$ in \eqref{nontrivial3}, we obtain $\epsilon\leq C_4(\mu_+(g(0))-\mu_+(g_E))^{\beta/2}$ and therefore, the Ricci flow is nontrivial.
For $T_i\leq t$, we have, by the Lojasiewicz-Simon\index{Lojasiewicz-Simon inequality} inequality,
\begin{align*}\left\|\tilde{g}^s_i(T_i)-\tilde{g}_i^s(t)\right\|_{C^{k-3}}\leq& C_4(\mu_+(\tilde{g}_i(t+t_i))-\mu_+(g_E))^{\theta}\\
                                                           \leq&C_4[-C_1t+ (\mu_+(\tilde{g}_i(t_i))-\mu_+(g_E))^{1-2\sigma}]^{-\frac{\theta}{2\sigma-1}}
                                                           \leq[-C_5t+C_6]^{-\frac{\theta}{2\sigma-1}}.
\end{align*}
Thus,
\begin{align*}\left\|g_E-\tilde{g}(t)\right\|_{C^{k-3}}\leq &\left\|g_E-\tilde{g}^s_i(T_i)\right\|_{C^{k-3}}+[-C_5t+C_6]^{-\frac{\theta}{2\sigma-1}}
      +\left\|\tilde{g}_i^s(t)-\tilde{g}(t)\right\|_{C^{k-3}}.
\end{align*}
It follows that $\left\|g_E-\tilde{g}(t)\right\|_{C^{k-3}}\to0$ as $t\to-\infty$. Therefore, $(\varphi_t^{-1})^*g(t)\to g_E$ in $C^{k-3}$ as $t\to-\infty$.
\end{proof}
\begin{rem}
Dynamical stability and instability under the volume-normalized Ricci flow follow from these theorems by projecting the flows above to the space of metrics of fixed volume and by a suitable rescaling of the time parameter. In this way, we obtain the results as stated in the introduction.
\end{rem}

\section{Einstein metrics with positive scalar curvature}\label{positivescalarcurvature}
In this section, we state analoguous stability/instability results for Einstein metrics with positive scalar curvature. Since the methods are very similar, we skip the details and we just explain the key steps and the main differences.
For details, we refer the reader to \cite[Chapter 6]{Kro13}.
We define the Ricci shrinker entropy\index{shrinker entropy} which was first introduced by G.\ Perelman in \cite{Per02}. Let \index{$\mathcal{W}_-(g,f,\tau)$}
\begin{align*}\mathcal{W}_-(g,f,\tau)=\frac{1}{(4\pi\tau)^{n/2}}\int_M [\tau(|\nabla f|^2_g+\scal_g)+f-n]e^{-f}\dv.
\end{align*}
For $\tau>0$, let\index{$\mu_-(g,\tau)$}
 \begin{align*}\mu_-(g,\tau)&=\inf \left\{\mathcal{W}_-(g,f,\tau)\suchthat{f\in C^{\infty}(M),\frac{1}{(4\pi\tau)^{n/2}}\int_M e^{-f}\dv_g=1}f\in C^{\infty}(M),\frac{1}{(4\pi\tau)^{n/2}}\int_M e^{-f}\dv_g=1\right\}.
 \end{align*}
For any $\tau>0$, the infimum is realized by a smooth function.
We define the shrinker entropy\index{shrinker entropy} as
\begin{align*}\nu_-(g)&=\inf \left\{\mu_-(g,\tau)\mid\tau>0\right\}.
\end{align*}
Observe that $\nu_-$ is scale and diffeomorphism invariant.
\index{nu@$\nu_-$-functional}\index{$\nu_-(g)$, shrinker entropy}If $\lambda(g)>0$\index{Lambda@$\lambda$-functional},
 then $\nu_-(g)$ is finite and realized by some $\tau_g>0$ (see \cite[Corollary 6.34]{CC07}).
In this case, a pair $(f_g,\tau_g)$ realizing $\nu_-(g)$ satisfies the equations\index{$f_g$, minimizer realizing $\nu_-(g)$}\index{$\tau_g$, minimizer realizing $\nu_-(g)$}
\begin{align}\label{nueulerlagrange}\tau(2\Delta f+|\nabla f|^2-\scal)-f+n+\nu_-&=0,\\
\label{nueulerlagrange1.5}\frac{1}{(4\pi\tau)^{n/2}}\int_M f e^{-f}\dv&=\frac{n}{2}+\nu_-,
\end{align}
see e.g. \cite[p.\ 5]{CZ12}. The first variation of $\nu_-$ is
\begin{align}\nu_-(g)'(h)=-\frac{1}{(4\pi\tau_g)^{n/2}}\int_M\left\langle \tau_g(\ric+\nabla^2 f_g)-\frac{1}{2}g,h\right\rangle e^{-f_g}\dv_g,
\end{align}
where $(f_g,\tau_g)$ realizes $\nu_-(g)$. Because of diffeomorphism invariance, $\nu_-$ is nondecreasing under the $\tau$-flow
\begin{align}\label{tauricciflow}\dot{g}(t)=-2\ric_{g(t)}+\frac{1}{\tau_{g(t)}}g(t).
\end{align}
\begin{rem}
The critical metrics of $\nu_-$ are precisely the shrinking gradient Ricci solitons. These are the metrics satisfying $\ric+\nabla^2 f=c g$ for some $f\in C^{\infty}(M)$ and $c>0$.
This includes all positive Einstein metrics. If $g_E$ is a positive Einstein metric with Einstein constant $\mu$, the pair $(f_{g_E},\tau_{g_E})$ satisfies
\begin{align}\label{coupledeulerlagrange}\tau_{g_E}=\frac{1}{2\mu},\qquad f_{g_E}=\log(\volume(M,g_E))-\frac{n}{2}(\log(2\pi)-\log(\mu)).
\end{align}
\end{rem}
\begin{prop}[Second variation of $\nu_-$]\label{nuhessian}\index{second variation!of $\nu_-$}The second variation of $\nu_-$ on a postive Einstein metric $(M,g_E)$ with constant $\mu$ is given by
$$\nu_-(g_E)''(h)=
\begin{cases} -\frac{1}{4\mu}\fint_M\langle \Delta_{E} h,h\rangle \dv,& \text{ if }\delta h=0\text{ and }\int_M\trace h\dv=0 ,\\
              0,& \text{ if }h\in \R\cdot g_E\oplus\delta^{*}(\Omega^1(M)).
\end{cases}$$
\end{prop}
\begin{proof}This is a simpler expression of the second variational formula in \cite[Theorem 2.1]{CHI04}.
By scale and diffeomorphism invariance\index{scale-invariance}\index{diffeomorphism!invariance}, $\nu_-(g_E)''$ vanishes on the subspace $\R\cdot g_E\oplus\delta^{*}(\Omega^1(M))$.
If $\delta h=0$ and $\int_M\trace h\dv=0$, the formula follows from the one in \cite{CHI04}.
Since $\Delta_E=\Delta_L-2\mu$, $\trace (\Delta_L h)=\Delta (\trace h)$ and $\delta(\Delta_L h)=\Delta_H(\delta h)$ \cite[pp. 28-29]{Lic61},
$\Delta_E$ preserves the above subspace. Here, $\Delta_H$ is the Hodge-Laplacian acting on one-forms. Thus, the splitting of above is orthogonal with respect to $\nu_-(g_E)''$.
Observe also that these two subspaces span the whole space $\Gamma(S^2M)$.
\end{proof}
Now, one has to check that on a small neighbourhood $\mathcal{U}$ of an Einstein metric, $f_g$ and $\tau_g$ are unique. Moreover, if $\mathcal{U}$ is small enough,
$\nu_-(g),f_g$ and $\tau_g$ depend analytically on the metric. These facts follows from
the the implicit function theorem for Banach manifolds and a bootstrap argument using elliptic regularity and the Euler-Lagrange equations \eqref{nueulerlagrange} and \eqref{nueulerlagrange1.5}.
Similar arguments were used in Lemma \ref{goodf_g} and Lemma \ref{analyticmu}.

Furthermore, one has to prove bounds for $\nu_-(g)$, $f_g$, $\tau_g$ and their derivatives as in Section \ref{technicalestimatesI}. These follow essentially from differentiating \eqref{nueulerlagrange} and \eqref{nueulerlagrange1.5} and using elliptic regularity.
Having developed these technical tools, one is able to prove
\begin{thm}\label{yamabemaximalityII}Let $(M,g_E)$ be a positive Einstein manifold with constant $\mu$ and let $\lambda$ be the smallest nonzero eigenvalue of the Laplacian
If $g_E$ is a local maximum of $\nu_-$, it is a local maximum of the Yamabe functional\index{Yamabe!functional} and we have $\lambda\geq 2\mu$. Conversely,
if $g_E$ is a local maximum of the Yamabe functional\index{Yamabe!functional} and $\lambda>2\mu$, then $g_E$ is a local maximum of $\nu_-$. In this case, any other local maximum is also an Einstein metric.
\end{thm}
\begin{proof}[Sketch of proof]We use the same notation as in the proof of Theorem \ref{yamabemaximality}.
Let $\bar{g}$ be a metric of constant scalar curvature. Then $\nu_-(\bar{g})$ is explicitly given by
 \begin{align*}\nu_-(\bar{g})=\log(\volume(M,\bar{g}))+\frac{n}{2}\log(\scal_{\bar{g}})+\frac{n}{2}(1-\log(2\pi n)).
\end{align*}
This follows from \eqref{nueulerlagrange} and \eqref{nueulerlagrange1.5} and the analytic dependence of $f_g$ and $\tau_g$ on $g$.
Recall that $g_E$ is a local maximum of the Yamabe functional if and only if it is a local maximum of the scalar curvature on  $\mathcal{C}_c$.
Therefore, $g_E$ is a local maximum of the Yamabe functional if and only if it is a local maximum of $\nu_-$ in $\mathcal{C}_c$ (where $c=\volume(M,g_E)$). By scale invariance, it is also a local maximum of $\nu_-$ in $\mathcal{C}$ in this case.
Let now $\bar{g}\in\mathcal{C}$ and $v\in C^{\infty}(M)$ such that $\int_M v\dv_{\bar{g}}=0$. Then by the first variational formula, $\frac{d}{dt}|_{t=0}\nu_-((1+tv)\bar{g})=0$.
A long but straightforward calculation shows that
\begin{align}\label{conformalsecondnu}\frac{d^2}{dt^2}\bigg\vert_{t=0}\nu_-((1+tv)\bar{g})=-\fint_M Lv\cdot v \dv,
\end{align}
where $L$ is the linear operator given by
\begin{align*}L=\frac{n+1}{4}\left(\frac{n}{\scal}\Delta-1\right)^{-1}\left(\frac{n}{\scal}\Delta-2\right)\left(\frac{n}{\scal}\Delta-\frac{n}{n-1}\right).
\end{align*}
This formula shows how the eigenvalue condition comes into play. Now the first assertion is clear, since $\lambda>\frac{\scal}{n-1}$ for any Einstein metric except the standard sphere \cite[Theorem 1 and Theorem 2]{Ob62}. The second assertion follows from the local decomposition of the space of metrics and Taylor expansion as in the proof of Theorem \ref{yamabemaximality}.
\end{proof}
\begin{cor}Let $(M,g_E)$ be a compact positive Einstein manifold with constant $\mu$. If $g_E$ is a local maximum of the Yamabe invariant and $\lambda>2\mu$,
any shrinking gradient Ricci soliton in a sufficiently small neighbourhood of $g_E$ is nessecarily Einstein.
\end{cor}
\begin{proof}This follows from Theorem \ref{yamabemaximalityII} and the fact that shrinking gradient Ricci solitons are precisely the critical\index{critical} points of $\nu_-$.
\end{proof}
\noindent
The proof of the following theorem is analoguous to the proof of Theorem \ref{nonintegrablegradient}.
\begin{thm}[Lojasiewicz-Simon inequality]\label{nonintegrablegradientII}\index{Lojasiewicz-Simon inequality}Let $(M,g_E)$ be a positive Einstein manifold. Then there exists a $C^{2,\alpha}$-neighbourhood $\mathcal{U}$
of $g_E$ and constants $\sigma\in[1/2,1)$, $C>0$ such that
\begin{align}\label{gradientnu2}|\nu_-(g)-\nu_-(g_E)|^{\sigma}\leq C \left\|\tau(\ric_g+\nabla^2 f_g)-\frac{1}{2}g\right\|_{L^2}
\end{align}
for all $g\in\mathcal{U}$.
\end{thm}
Now we have the tools to prove the stability/instability results for positively curved Einstein metrics. The proofs are the same as in the negative case.
\begin{thm}[Dynamical stability]\label{dynamicalstabilitymodulodiffeo}\index{stable!dynamically}\index{modulo diffeomorphism}Let $(M,g_E)$ be a compact positive Einstein manifold with constant $\mu$ and let $k\geq3$. Suppose that $g_E$ is a local maximizer of the Yamabe functional
\index{Yamabe!functional}and the smallest nonzero eigenvalue of the Laplacian is larger than $2\mu$. Then for every $C^{k}$-neighbourhood $\mathcal{U}$
of $g_E$, there exists a $C^{k+2}$-neighbourhood $\mathcal{V}$ such that the following holds:

For any metric $g_0\in\mathcal{V}$, there exists a $1$-parameter family\index{1@$1$-parameter family} of diffeomorphisms $\varphi_t$ and a positive function $v$ such that for the $\tau$-flow $g(t)$
 starting at $g_0$, the modified flow $\varphi_t^*g(t)$
stays in $\mathcal{U}$ for all time and converges to an Einstein metric $g_{\infty}$ in $\mathcal{U}$ as $t\to\infty$.
The convergence is of polynomial rate, i.e.\ there exist constants $C,\alpha>0$ such that
\begin{align*}\left\|\varphi_t^*g(t)-g_{\infty}\right\|_{C^k}\leq C(t+1)^{-\alpha}.
\end{align*}
\end{thm}
\begin{thm}[Dynamical instability]\label{dynamicalinstabilitymodulodiffeo}\index{unstable!dynamically}\index{modulo diffeomorphism}Let $(M,g_E)$ be a positive Einstein manifold that is not a local maximizer of $\nu_-$. 
 Then there exists a nontrivial ancient $\tau$-flow\index{ancient} $g(t)$,
 $t\in (-\infty,0]$ and a $1$-parameter family\index{1@$1$-parameter family} of diffeomorphisms $\varphi_t$, $t\in (-\infty,0]$ such that $\varphi_t^*g(t)\to g_E$ as $t\to\infty$.
\end{thm}
\begin{rem}
 One gets stronger convergence statements if one replaces the assumption of local maximality of the Yamabe functional in the Theorems \ref{negativemodulostability} and \ref{dynamicalstabilitymodulodiffeo} by the assumption that the Einstein operator $\Delta_E$ is nonnegative and all infinitesimal Einstein deformations are integrable
 i.e.\ all elements in the kernel of $\Delta_E$ can be integrated to curves of Einstein metrics.
 In this case, the flow will converge exponentially and we do not have to pull back the flow by diffeomorphisms.
 For details, see \cite[Section 5.4 and 6.4]{Kro13}.
\end{rem}

\section{Dynamical instability of the complex projective space}\label{complexinstability}\index{complex projective space}
In order to prove Theorem \ref{Thm3}, we show that the Einstein metric cannot be a local maximum of $\nu_-$. Since the second variation of $\nu_-$ at $g_E$ may be nonpositive, 
we have to compute a third variation of it.
\begin{prop}\label{thirdmuexplicit}Let $(M,g_E)$ be a positive Einstein manifold with constant $\mu$ and suppose we have a function $v\in C^{\infty}(M)$ such that $\Delta v=2\mu\cdot v$. Then the third variation of $\nu_-$\index{third variation!of $\nu_-$}
 in the direction of $v\cdot g_E$ is given by
\begin{align*}\frac{d^3}{dt^3}\bigg\vert_{t=0}\nu_-(g_E+tv\cdot g_E)=(n-2)\fint_M v^3 \dv.
\end{align*}
\end{prop}
\begin{proof}Put $u=\frac{e^{-f}}{(4\pi\tau)^{n/2}}$. By the first variation, 
the negative of the $L^2(u\dv)$-gradient\index{gradient!$L^2$}  of $\nu_-$ is given by $\nabla\nu_-=\tau(\ric+\nabla^2 f)-\frac{g}{2}$, so
\begin{align*}\frac{d}{dt}\bigg\vert_{t=0}\nu_-(g_E+th)=-\int_M \langle\nabla\nu_-,h\rangle u \dv.
\end{align*}
Since $(M,g_E)$ is a critical\index{critical} point of $\nu_-$, we clearly have $\nabla\nu_-=0$.
Since $v$ is a nonconstant eigenfunction, $\int_M v\dv=0$. Thus
by \cite[Lemma 2.4]{CZ12}, $\tau'$ vanishes.
Recall from \eqref{coupledeulerlagrange} that $\tau_{g_E}=\frac{1}{2\mu}$ and $f_{g_E}$ is constant. Therefore, by the first variation of the Ricci tensor\index{first variation!of the Ricci tensor},
\begin{align*}\nabla\nu'_-=&\tau'\mu g_E+\frac{1}{2\mu}(\ric'+\nabla^2 (f'))-\frac{g'}{2}\\
                          =&\frac{1}{2\mu}\left(\frac{1}{2}\Delta_L(v\cdot g_E)-\delta^*\delta(v\cdot g_E)-\frac{1}{2}\nabla^2\trace (v\cdot g_E)+\nabla^2 (f')\right)-\frac{v\cdot g_E}{2}\\
                          =&\frac{1}{2\mu}\left(\left(1-\frac{n}{2}\right)\nabla^2 v+\nabla^2 (f')\right).
\end{align*}
To compute $f'$, we consider the Euler-Lagrange equation\index{Euler-Lagrange equation}
\begin{align}\label{lasteulerlagrange}\tau(2\Delta f+|\nabla f|^2-\scal)-f+n+\nu_-=0.
\end{align}
By differentiating once and using $\tau'=0$ and $\nu_-'=0$,
\begin{align*}\frac{1}{2\mu}(2\Delta f'-\scal')-f'=0,
\end{align*}
and by the first variation of the scalar curvature\index{first variation!of the scalar curvature},
\begin{align*}\left(\frac{1}{\mu}\Delta-1\right)f'=\frac{1}{2\mu}\scal'
                                                      =&\frac{1}{2\mu}((n-1)\Delta v-n\mu v).
\end{align*}
By a well-known eigenvalue estimate for the Laplacian (\cite[Theorem 1]{Ob62})\index{Obata's eigenvalue estimate}, $\frac{1}{\mu}\Delta-1$ is invertible.
By using the eigenvalue equation, we therefore obtain 
\begin{align}\label{thef'}f'=\left(\frac{n}{2}-1\right)v.
\end{align}
Thus,
\begin{align}\label{nu''=0}\nabla\nu'_-=0,
\end{align}
and therefore, the third variation equals\index{third variation!of $\nu_-$}
\begin{align*}\frac{d^3}{dt^3}\bigg\vert_{t=0}\nu_-(g_E+tv\cdot g_E)=-\int_M \langle\nabla\nu''_-,v\cdot g_E\rangle u \dv.
\end{align*}
Since $\tau_{g_E}=\frac{1}{2\mu}$ and $\tau'=0$,
\begin{align*}\nabla\nu''_-=\tau''\mu\cdot g_E+\frac{1}{2\mu}(\ric+\nabla^2f)''.
\end{align*}
The function $u$ is constant since $f$ is constant. Thus, the $\tau''$-term drops out after integration. We are left with
\begin{align}\label{nu'''}\frac{d^3}{dt^3}\bigg\vert_{t=0}\nu_-(g_E+tv\cdot g_E)=-\frac{1}{2\mu}\int_M \langle(\ric+\nabla^2f)'',v\cdot g_E\rangle u \dv.
\end{align}
We first compute $\ric''$. Let $g_t=(1+tv)g_E$ and $v_t=\frac{v}{1+tv}$. Then $g'_t=v_t\cdot g_t$ and $\frac{d}{dt}|_{t=0}v_t=-v^2$.
By the first variation of the Ricci tensor\index{first variation!of the Ricci tensor},
\begin{align*}\frac{d}{dt}\ric_{g_t}             =\frac{1}{2}[(\Delta v_t)g_t-(n-2)\nabla^2 v_t],
\end{align*}
and the second variation at $g_E$\index{second variation!of the Ricci tensor} is equal to
\begin{align*}\frac{d^2}{dt^2}\bigg\vert_{t=0}\ric_{g_E+tv\cdot g_E}=&\frac{d}{dt}\bigg\vert_{t=0}\frac{1}{2}[(\Delta v_t)g_t-(n-2)\nabla^2 v_t]\\
=&\frac{1}{2}[(\Delta' v+ \Delta (v')+\Delta v\cdot v)g_E-(n-2)(\nabla^2)'v-(n-2)\nabla^2 (v')]\\
=&\frac{1}{2}[(\langle v\cdot g_E,\nabla^2 v\rangle-\langle\delta(v\cdot g_E)+\frac{1}{2}\nabla\trace(v\cdot g_E),\nabla v\rangle)g_E\\
 &+(-\Delta v\cdot v+2|\nabla v|^2)g_E-(n-2)\left(\frac{1}{2}|\nabla v|^2g_E-\nabla v\otimes\nabla v\right)\\
 &+(n-2)(2\nabla^2v\cdot v+2\nabla v\otimes\nabla v)]\\
=&-\left(\frac{n}{2}-2\right)|\nabla v|^2g_E-2\mu v^2g_E+3\left(\frac{n}{2}-1\right)\nabla v\otimes\nabla v+(n-2)\nabla^2 v\cdot v,
\end{align*}
where we used the first variational formulas of the Laplacian\index{first variation!of the Laplacian} and the Hessian\index{first variation!of the Hessian} in Lemma \ref{firsthessian}.
Let us now compute the $(\nabla^2f)''$-term. Since $f_{g_E}$ is constant,
\begin{align*}\frac{d^2}{dt^2}\bigg\vert_{t=0}\nabla^2f_{g_E+tv\cdot g_E}&=\nabla^2 (f'')+2(\nabla^2)'f'
                                                            =\nabla^2 (f'')-\nabla v\otimes\nabla f'-\nabla f'\otimes\nabla v+\langle \nabla f',\nabla v\rangle g_E.
\end{align*}
We already know that $f'=(\frac{n}{2}-1)v$ by \eqref{thef'}.
To compute $f''$, we differentiate (\ref{lasteulerlagrange}) twice. By \eqref{nu''=0}, $\nu_-''=0$.
Since also $\tau'=0$ as remarked above, we obtain
\begin{equation}\begin{split}\label{lasteulerlagrange''}0&=-\tau''\scal+\tau(2\Delta f+|\nabla f|^2-\scal)''-f''\\
               &=-\tau''n\mu+\frac{1}{\mu}\Delta f''+\frac{2}{\mu}\Delta'f'+\frac{1}{\mu}|\nabla (f')|^2-\frac{1}{2\mu}\scal''-f''.
\end{split}\end{equation}
Because $\Delta v=2\mu v$,
\begin{equation}\begin{split}\label{Delta'f'}\Delta'f'=&\langle v\cdot g,\nabla^2 f'\rangle-\left\langle \delta(v\cdot g)+\frac{1}{2}\nabla \trace (v\cdot g),\nabla f'\right\rangle\\
                       \overset{\eqref{thef'}}{=}
                       &\left(\frac{n}{2}-1\right)\left[-2\mu v^2-\left(\frac{n}{2}-1\right)|\nabla v|^2\right].
\end{split}
\end{equation}
Next, we compute $\scal''$. As above, let $g_t=(1+tv)g_E$ and $v_t=\frac{v}{1+tv}$. Then by the first variation of the scalar curvature,
\begin{align*}\frac{d}{dt}\scal_{g_t}=
                                     (n-1)\Delta v_t-\scal_{g_t}v_t.
\end{align*}
The second variation of the scalar curvature\index{second variation!of the scalar curvature} at $g_E$ is equal to
\begin{align*}\frac{d^2}{dt^2}\bigg\vert_{t=0}\scal_{g_E+tv\cdot g_E}=&\frac{d}{dt}\bigg\vert_{t=0}[(n-1)\Delta v_t-\scal_{g_t}v_t]\\
      =&(n-1)[\Delta'v+\Delta(v')]-n\mu\cdot v'-\scal'v\\
      =& (n-1)[\langle v\cdot g_E,\nabla^2 v\rangle-\langle\delta(v\cdot g_E)+\frac{1}{2}\nabla \trace(v\cdot g_E),\nabla v\rangle-\Delta (v^2)]\\
       &+n\mu\cdot v^2-[\Delta\trace(v\cdot g_E)+\delta\delta(v\cdot g_E)-\langle\ric,v\cdot g_E\rangle]v\\
      =&-(n-1)\left(\frac{n}{2}-3\right)|\nabla v|^2+2\mu (4-3n)\cdot v^2.
\end{align*}
By \eqref{thef'}, $|\nabla (f')|^2=(\frac{n}{2}-1)^2|\nabla v|^2$. Thus, we can rewrite (\ref{lasteulerlagrange''}) as
\begin{align*}\left(\frac{1}{\mu}\Delta-1\right)f''=&\tau''n\mu-\frac{1}{\mu}(2\Delta'f'+|\nabla (f')|^2-\frac{1}{2}\scal'')\\
                                        \overset{\eqref{Delta'f'}}{=}&\tau''n\mu-\frac{1}{\mu}[-2(n-2)\mu v^2-2\left(\frac{n}{2}-1\right)^2|\nabla v|^2+\left(\frac{n}{2}-1\right)^2|\nabla v|^2\\
                                         &+\frac{n-1}{2}\left(\frac{n}{2}-3\right)|\nabla v|^2+(3n-4)\mu  v^2]\\
                                        =&\tau''n\mu-\frac{1}{\mu}\left[n\mu v^2+\left(-\frac{3}{4}n+\frac{1}{2}\right)|\nabla v|^2\right]=:(A).
\end{align*}
Since $\frac{1}{\mu}\Delta-1$ is invertible, we can rewrite the above as
\begin{align*}f''=(\frac{1}{\mu}\Delta-1)^{-1}(A).
\end{align*}
By integrating,
\begin{align*}-\frac{1}{2\mu}\int_M& \langle(\nabla^2f)'',v\cdot g_E\rangle u \dv=-\frac{1}{2\mu}\fint_M\langle(\nabla^2f)'',v\cdot g_E\rangle \dv\\
                                  =&-\frac{1}{2\mu}\fint_M\langle\nabla^2 (f'')-\nabla v\otimes\nabla f'-\nabla f'\otimes\nabla v+\langle \nabla f',\nabla v\rangle g_E,v\cdot g_E\rangle \dv\\
                                  \overset{\ref{thef'}}{=}&-\frac{1}{2\mu}\fint_M\left\langle\nabla^2 (f'')-(n-2)\nabla v\otimes\nabla v+\left(\frac{n}{2}-1\right)|\nabla v|^2 g_E,v\cdot g_E\right\rangle \dv\\
                                  =&-\frac{1}{2\mu}\fint_M \left[-\Delta(f'')v+\frac{1}{2}(n-2)^2|\nabla v|^2v\right]\dv\\
                                  =&-\frac{1}{2\mu}\fint_M \left[-(A)\left(\frac{1}{\mu}\Delta-1\right)^{-1}\Delta v+\frac{1}{2}(n-2)^2|\nabla v|^2v\right]\dv\\
                                  =&-\frac{1}{2\mu}\fint_M \left[-2\mu(A)v+\frac{1}{2}(n-2)^2|\nabla v|^2v\right]\dv.
\end{align*}
Now we insert the definition of $(A)$. Since the term containing $\tau''$ drops out after integration, we are left with
\begin{align*}-\frac{1}{2\mu}\int_M& \langle(\nabla^2f)'',v\cdot g_E\rangle u \dv=
                                   -\frac{1}{2\mu}\fint_M \left[(2n\mu v^3+\frac{1}{2}(n^2-7n+6)|\nabla v|^2v\right]\dv.
\end{align*}
By the second variation of the Ricci tensor\index{second variation!of the Ricci tensor} computed above,
\begin{align*}-\frac{1}{2\mu}\int_M \langle\ric'',v\cdot g_E\rangle u \dv=&-\frac{1}{2\mu}\fint_M \langle\ric'',v\cdot g_E\rangle \dv\\
=&-\frac{1}{2\mu}\fint_M[-n\left(\frac{n}{2}-2\right)|\nabla v|^2v\\
 &-2\mu n v^3+3\left(\frac{n}{2}-1\right)|\nabla v|^2v-(n-2)\Delta v\cdot v^2]\dv\\
=&-\frac{1}{2\mu}\fint_M\left[\left(-\frac{n^2}{2}+\frac{7n}{2}-3\right)|\nabla v|^2v-4(n-1)\mu v^3\right]\dv.
\end{align*}
Adding up these two terms, we obtain
\begin{align*}\frac{d^3}{dt^3}\bigg\vert_{t=0}\nu_-(g+tv\cdot g)\overset{\eqref{nu'''}}{=}-\frac{1}{2\mu}\fint_M (4-2n)\mu v^3\dv.
\end{align*}
% By integration by parts\index{integration by parts},
% \begin{align*}\int_M |\nabla v|^2v\dv=\frac{1}{2}\int_M\Delta v\cdot v^2\dv=\mu\int_M v^3\dv,
% \end{align*}
and therefore, we finally have
\begin{align*}\frac{d^3}{dt^3}\bigg\vert_{t=0}\nu_-(g+tv\cdot g)=(n-2)\fint_M v^3\dv,
\end{align*}
which finishes the proof.
\end{proof}
\begin{cor}\label{thirdpowerlaplacian}Let $(M^n,g_E)$, $n\geq3$ be a positive Einstein manifold with constant $\mu$. Suppose there exists a function $v\in C^{\infty}(M)$ such that $\Delta v=2\mu v$ and $\int_M v^3\dv\neq0$.
Then $g_E$ is not a local maximum of $\nu_-$.
\end{cor}
\begin{proof}Let $\varphi(t)=\nu_-(g_E+tv\cdot g_E)$. By the proof of the proposition above, $\varphi'(0)=0$, $\varphi''(0)=0$ and $\varphi'''(0)\neq 0$.
Depending on the sign of the third variation, $\varphi(t)>\varphi(0)$ either for $t\in (-\epsilon,0)$ or $t\in (0,\epsilon)$. This proves the assertion.
\end{proof}
\noindent
Theorem \ref{Thm3} now follows from Corollary \ref{thirdpowerlaplacian} and Theorem \ref{dynamicalinstabilitymodulodiffeo}.
 \begin{proof}[Proof of Corollary \ref{unstableCPn}]Let $\mu$ be the Einstein constant. We prove the existence of a function $v\in C^{\infty}(\CP^n)$ satisfying $\Delta v=2\mu v$ and $\int_{\CP^n} v^3\dv\neq 0$.
First, we rewiev the construction of eigenfunctions on $\CP^n $ as explained in \cite[Section III C]{BGM71}.
Consider $\C^{n+1}=\R^{2n+2}$ with coordinates $(x_1,\ldots,x_{n+1},y_1,\ldots,y_{n+1})$ and let $z_j=x_j+iy_j$, $\bar{z}_j=x_j-iy_j$
be the complex coordinates\index{complex coordinates}. Defining\index{$\partial_{x_i}$, directional derivative} $\partial_{z_j}=\frac{1}{2}(\partial_{x_j}-i\partial_{y_j})$ and $\partial_{\bar{z}_j}=\frac{1}{2}(\partial_{x_j}-i\partial_{y_j})$,
we can rewrite the Laplace operator on $\C^{n+1}$ as
\begin{align*}\Delta=-4\sum_{j=1}^{n+1}\partial_{z_j}\circ \partial_{\bar{z}_j}.
\end{align*}
Let $P_{k,k}$\index{$P_{k,k}$, set of homogeneous polnomials of degree $k$ in $z$ and $\bar{z}$}\index{$H_{k,k}$, set of harmonic polynomials in $P_{k,k}$} be the space of complex polynomials\index{polynomial} on $\C^{n+1}$ which are homogeneous of degree $k$ in $z$ and $\bar{z}$ and let $H_{k,k}$ the subspace of harmonic polynomials\index{polynomial!harmonic} in $P_{k,k}$. We have
\begin{align*}P_{k,k}=H_{k,k}\oplus r^2P_{k-1,k-1}.
\end{align*}
Elements in $P_{k,k}$ are $S^1$-invariant and thus, they descend to functions on the quotient $\CP^n=S^{2n+1}/S^1$. The eigenfunctions to the k-th eigenvalue of the Laplacian on $\CP^n$ (where $0$ is meant to be the
$0$-th eigenvalue) are precisely the restrictions of functions in $H_{k,k}$. Since $2\mu$ is the first nonzero eigenvalue, its eigenfunctions are restrictions of functions in $H_{1,1}$.

Let $h_1(z,\bar{z})=z_1\bar{z}_2+z_2\bar{z}_1$, $h_2(z,\bar{z})=z_2\bar{z}_3+z_3\bar{z}_2$, $h_3(z,\bar{z})=z_3\bar{z}_1+z_1\bar{z}_3$ and let $v$ be the eigenfunction which is
 the restriction of $h=h_1+h_2+h_3\in H_{1,1}$.  Note that $h$ is real-valued and so is $v$. Then $v^3$ is the restriction of 
\begin{align}\label{polynomialsplitting}h^3\in P_{3,3}=H_{3,3}\oplus r^2H_{2,2}\oplus r^4 H_{1,1}\oplus r^6 H_{0,0}.
\end{align}
We show that $\int_{S^{2n+1}}h^3\dv\neq0$. At first,
\begin{align*}h^3=\sum_{j=1}^3 h_j^3+3\sum_{j\neq l}h_j\cdot h_l^2+6h_1\cdot h_2\cdot h_3.
\end{align*}
Note that $\int_{S^{2n+1}}h_1^3\dv=0$ because $h_1$ is antisymmetric with respect to the isometry\index{isometry} $(z_1,\bar{z}_1)\mapsto(-z_1,-\bar{z}_1)$.
For the same reason, $\int_{S^{2n+1}}h_1\cdot h_2^2\dv=0$. Similarly, we show that all other terms of this form vanish after integration so it remains to deal with the last term of above. We have
\begin{align*}h_1\cdot h_2\cdot h_3(z,\bar{z})=2|z_1|^2|z_2|^2|z_3|^2+\sum_{\sigma\in S_3}|z_{\sigma(1)}|^2z_{\sigma(2)}^2\bar{z}_{\sigma(3)}^2.
\end{align*}
Consider $|z_1|^2z_2^2\bar{z}_3^2$. This polynomial is antisymmetric with respect to the isometry\index{isometry} $(z_2,\bar{z}_2)\mapsto (i\cdot z_2,i\cdot\bar{z}_2)$ and therefore,
\begin{align*}\int_{S^{2n+1}}|z_1|^2z_2^2\bar{z}_3^2\dv=0.
\end{align*}
Similarly, we deal with the other summands. In summary, we have
\begin{align*}\int_{S^{2n+1}}h^3\dv=6\int_{S^{2n+1}}h_1\cdot h_2\cdot h_3\dv=12\int_{S^{2n+1}}|z_1|^2|z_2|^2|z_3|^2\dv>0,
\end{align*}
since the integrand on the right hand side is nonnegative and not identically zero. We decompose $h^3=\sum_{j=0}^3h_j$, where $h_j\in r^{6-2j}H_{j,j}$.
Since the restrictions of the $h_j$ to $S^{2n+1}$ are eigenfunctions to the $2k$-th eigenvalue of the Laplacian on $S^{2n+1}$ (see \cite[Section III C]{BGM71}), we have that $h_0\neq0$ because the integral is nonvanishing.
This decomposition induces a decomposition of $v^3=\sum_{i=0}^3v_i$ where $v_i$ is an eigenfunction of the i-th eigenvalue of $\Delta_{\CP^n}$ and $v_0\neq0$. Therefore,
$\int_{\CP^n} v^3\dv\neq 0$. The assertion follows from Theorem \ref{Thm3}.
\end{proof}
\appendix
\section{Variational formulas}
Here we prove some variational formulas needed throughout the text.
 \begin{lem}\label{firstcuvature}Let $(M,g)$ be Riemannian manifold and denote the \index{first variation!of the Levi-Civita connection}\index{connection!Levi-Civita}first variation of the Levi-Civita connection in the direction of $h$ by $G$.
Then $G$ is a $(1,2)$ tensor field, given by\index{$G$, first variation of the Levi-Civita connection}
\begin{align*}g(G(X,Y),Z)=\frac{1}{2}(\nabla_Xh(Y,Z)+\nabla_Yh(X,Z)-\nabla_Zh(X,Y)).
\end{align*}
The first variation of the Riemann tensor, the Ricci tensor\index{first variation!of the Ricci tensor} and the scalar curvature
\index{first variation!of the scalar curvature} are given by
\begin{align*}
 \frac{d}{dt}\bigg\vert_{t=0}R_{g+th}(X,Y,Z,W)=&\frac{1}{2}(\nabla^2_{X,Z}h(Y,W)+\nabla^2_{Y,W}h(X,Z)-\nabla^2_{Y,Z}h(X,W)\\
                                      &-\nabla^2_{X,W}h(Y,Z)+h(R_{X,Y}Z,W)-h(Z,R_{X,Y}W)),\\
\frac{d}{dt}\bigg\vert_{t=0}\ric_{g+th}(X,Y)=&\frac{1}{2}\Delta_Lh(X,Y)-\delta^{*}(\delta h)(X,Y)-\frac{1}{2}\nabla^2_{X,Y}\trace h,\\
\frac{d}{dt}\bigg\vert_{t=0}\scal_{g+th}=&\Delta_g(\trace_g h)+\delta_g(\delta_g h)-\langle \ric_g,h\rangle_g.
\end{align*}
Furthermore, the first variation of the volume element\index{volume!element} is given by
\begin{equation*}\frac{d}{dt}\bigg\vert_{t=0}\dv_{g+th}=\frac{1}{2}\trace_g h\cdot \dv_g.
\end{equation*}
 \end{lem}
\begin{proof}
 See \cite[Theorem 1.174]{Bes08} and \cite[Proposition 1.186]{Bes08}.
\end{proof}
\begin{lem}\label{firsthessian}The first variation of the Hessian\index{first variation!of the Hessian} and the Laplacian\index{first variation!of the Laplacian} are given by
\begin{align*}\frac{d}{dt}\bigg\vert_{t=0}{}^{g+th}\nabla_{X,Y}^2f=&-\frac{1}{2}[\nabla_Xh(Y,\gradient f)+\nabla_Yh(X,\gradient f)-\nabla_{\gradient f}h(X,Y)],\\
              \frac{d}{dt}\bigg\vert_{t=0}\Delta_{g+th} f=&\langle h,\nabla^2 f\rangle-\left\langle\delta h+\frac{1}{2}\nabla\trace h,\nabla f\right\rangle.
\end{align*}
 \end{lem}
\begin{proof}
We use local coordinates.
Let $f$ be a smooth function. Then the first variation of the Hessian is
\begin{align*}\frac{d}{dt}\bigg\vert_{t=0}(\nabla^2_{ij}f)=&\frac{d}{dt}\bigg\vert_{t=0}(\partial^2_{ij}-\Gamma_{ij}^k\partial_kf)
                                                 =-\frac{1}{2}g^{kl}(\nabla_ih_{jl}+\nabla_jh_{il}-\nabla_lh_{ij})\partial_kf
\end{align*}
by the first variation of the Levi-Civita connection\index{first variation!of the Levi-Civita connection}.
The first variation of the Laplacian is
\begin{align*}\frac{d}{dt}\bigg\vert_{t=0}(\Delta f)&=-\frac{d}{dt}\bigg\vert_{t=0}(g^{ij}\nabla^2_{ij}f)
                                        =h^{ij}\nabla^2_{ij}f-g^{kl}(\delta h_l+\frac{1}{2}\nabla_l\trace h)\partial_kf.
\qedhere
\end{align*}
\end{proof}
\begin{lem}\label{secondhessian}The second variations of the Hessian\index{second variation!of the Hessian}, the Laplacian\index{second variation!of the Laplacian}, the Ricci tensor\index{second variation!of the Ricci tensor} and the scalar curvature
\index{second variation!of the scalar curvature} have the schematic expressions
\begin{align*}\frac{d}{ds}\frac{d}{dt}\bigg\vert_{s,t=0}\nabla_{g+sk+th}^2f=&k*\nabla h*\nabla f+\nabla k*h*\nabla f,\\
          \frac{d}{ds}\frac{d}{dt}\bigg\vert_{s,t=0}\Delta_{g+sk+th} f=&k*\nabla h*\nabla f+\nabla k*h*\nabla f,\\
           \frac{d}{ds}\frac{d}{dt}\bigg\vert_{s,t=0}\ric_{g+sk+th}=&k*\nabla^2h+\nabla^2k*h+\nabla k*\nabla h+R*k*h,\\
                  \frac{d}{ds}\frac{d}{dt}\bigg\vert_{s,t=0}\scal_{g+sk+th}=&k*\nabla^2h+\nabla^2k*h+\nabla k*\nabla h+R*k*h.
\end{align*}
Here, $*$ is Hamilton's notation for a combination of tensor products with contractions.
\end{lem}
\begin{proof}By the first variation of the Christoffel symbols it is not hard to see that the first two covariant derivatives of a $(0,2)$-tensor $h$ can be written as
\begin{align}\label{cov1}\frac{d}{dt}\bigg\vert_{t=0}\nabla_{g+tk} h=&k*\nabla h+\nabla{k}*h,\\
              \label{cov2}     \frac{d}{dt}\bigg\vert_{t=0}\nabla^2_{g+tk} h=&k*\nabla^2h+\nabla^2k*h+\nabla k*\nabla h.
\end{align}
The first variation of the Hessian is of the schematic form $\nabla h*\nabla f$
and therefore,
\begin{align*}
 \frac{d}{ds}\frac{d}{dt}\bigg\vert_{s,t=0}\nabla_{g+sk+th}^2f= \frac{d}{ds}\bigg\vert_{s=0}(\nabla h*\nabla f)=&k*\nabla h*\nabla f+(\frac{d}{ds}\bigg\vert_{s=0}\nabla h)*\nabla f\\
=&k*\nabla h*\nabla f+\nabla k*h*\nabla f.
\end{align*}
The expression for the second variation of the Laplacian is shown similarly.
By Lemma \ref{firstcuvature}, the first variational formulas for the Riemann curvature tensor\index{first variation!of the Riemann curvature tensor} and the Ricci tensor\index{first variation!of the Ricci tensor} are of the form
\begin{align}\label{curv1}\frac{d}{dt}\bigg\vert_{t=0}R&=(\nabla^2*h)+(R*h),\\
              \label{curv2}\frac{d}{dt}\bigg\vert_{t=0}\ric&=(\nabla^2*h)+(R*h).
\end{align}
The second variational expression of the Ricci tensor now follows from differentiating \eqref{curv2} in the direction of $k$ and using $\eqref{cov1}$ and $\eqref{curv1}$.
The last expression for the scalar curvature is shown similarly.
\end{proof}
%  \begin{tiny}
%  \textcolor{white}{
%  \section{References}}
%  \end{tiny}
\newcommand{\etalchar}[1]{$^{#1}$}

\end{document}